\documentclass[10pt,twoside,a4paper]{article}

\usepackage[american]{babel}
\usepackage[
			hmargin=1.89cm,
			vmargin=1.89cm,
			nomarginpar]
			{geometry}
\usepackage[fleqn]{amsmath}
\usepackage{amssymb,amsthm}
\usepackage{xspace}
\usepackage[dvipsnames]{xcolor}
\usepackage{dsfont}
\usepackage{microtype}
\usepackage{appendix}
\DisableLigatures[f]{encoding = *}

\definecolor{webgreen}{rgb}{0,.5,0}
\definecolor{webbrown}{rgb}{.6,0,0}
\usepackage[colorlinks=true,urlcolor=webbrown, linkcolor=RoyalBlue, citecolor=webgreen,pdfauthor={Alessandro Zocca}]{hyperref}

\usepackage{tabularx} 
\usepackage{caption}
\usepackage{subfig}  
\usepackage{placeins}
\usepackage{psfrag}
\usepackage{graphicx}

\newcommand{\ie}{i.e.,~}
\newcommand{\eg}{e.g.,~}

\newcommand{\ed}{\,{\buildrel d \over =}\,}
\newcommand{\cd}{\xrightarrow{d}}
\newcommand{\E}{\mathbb E}
\newcommand{\pr}[1]{\mathbb P \Big ( #1 \Big )}
\newcommand{\prin}[1]{\mathbb P ( #1 )}
\newcommand{\st}{\mathbin{\lvert}}

\newcommand{\N}{\mathbb N}
\newcommand{\R}{\mathbb R}

\renewcommand{\a}{\alpha}
\renewcommand{\b}{\beta}

\newcommand{\e}{\varepsilon}
\newcommand{\h}{\eta}
\renewcommand{\o}{\omega}
\renewcommand{\t}{\tau}
\renewcommand{\l}{\lambda}
\newcommand{\s}{\sigma}

\newcommand{\G}{\Gamma}

\renewcommand{\L}{\Lambda}

\newcommand{\cX}{\mathcal{X}}

\newcommand{\LT}{\mathcal{L}}
\renewcommand{\ss}{\cX^s}

\newcommand{\binf}{\b \to \infty}

\newcommand{\limb}{\lim_{\binf}}

\newcommand{\rmexp}{\mathrm{Exp}}
\newcommand{\rmgeo}{\mathrm{Geo}}

\renewcommand{\DH}{\Delta H}
\renewcommand{\aa}{\mathbf{a}}
\newcommand{\bb}{\mathbf{b}}
\newcommand{\cc}{\mathbf{c}}

\renewcommand{\AA}{V_\aa}
\newcommand{\BB}{V_\bb}
\newcommand{\CC}{V_\cc}
\newcommand{\vdtr}{K}
\newcommand{\hdtr}{L}

\newcommand{\tab}{\t^{\aa}_{\bb}}
\newcommand{\tabc}{\t^{\aa}_{\{\bb,\cc\}}}
\newcommand{\xtbb}{\{X^\b_t \}_{t \in \N}}

\newcommand{\tha}{\tau^\s_{A}}

\newtheorem{thm}{Theorem}[section]
\newtheorem{cor}[thm]{Corollary}
\newtheorem{lem}[thm]{Lemma}
\newtheorem{prop}[thm]{Proposition}
\theoremstyle{definition}
\newtheorem{rem}{Remark}

\title{Tunneling of the hard-core model on finite triangular lattices} 
\date{\today}
\author{Alessandro~Zocca\thanks{California Institute of Technology, Pasadena, California, US. Email: \texttt{azocca@caltech.edu}}}

\begin{document}
\maketitle

\begin{abstract}
\noindent We consider the hard-core model on finite triangular lattices with Metropolis dynamics. Under suitable conditions on the triangular lattice sizes, this interacting particle system has three maximum-occupancy configurations and we investigate its high-fugacity behavior by studying tunneling times, \ie the first hitting times between these maximum-occupancy configurations, and the mixing time. The proof method relies on the analysis of the corresponding state space using geometrical and combinatorial properties of the hard-core configurations on finite triangular lattices, in combination with known results for first hitting times of Metropolis Markov chains in the equivalent zero-temperature limit. In particular, we show how the order of magnitude of the expected tunneling times depends on the triangular lattice sizes in the low-temperature regime and prove the asymptotic exponentiality of the rescaled tunneling time leveraging the intrinsic symmetry of the state space.\\

\end{abstract}


\section{Introduction}
\label{sec1}
The \textit{hard-core model} was introduced in the chemistry and statistical physics literature to describe the behavior of a gas whose particles have non-negligible radii and cannot overlap~\cite{GF65,RC66,vdBS94}. 

A finite undirected graph $\L = (V,E)$ describes the spatial structure of the finite volume in which the particles interact. More specifically, the vertices represent the possible sites where particles can reside, while the hard-core constraints are represented by edges connecting the pairs of sites that cannot be occupied simultaneously. Particle configurations that do not violate these hard-core constraints are then in one-to-one correspondence with the independent sets of the graph $\L$, whose collection we denote by $\mathcal I(\L)$. 
Given $\l >0$, the \textit{hard-core measure with activity} (or \textit{fugacity}) $\l$ is the probability measure on $\mathcal I(\L)$ defined by
\begin{equation}
\label{eq:hcgb}
	\pi_{\l}(I) := \frac{\l^{|I|}}{Z_{\l}(\L)}, \quad I \in \mathcal I(\L),
\end{equation}
where $Z_{\l}(\L)$ is the appropriate normalizing constant, also called \textit{partition function}. 


In this paper we focus on the \textit{dynamics} of particles with hard-core repulsion on \textit{finite} graphs. The evolution over time of this interacting particle system is described by a reversible single-site update Markov chain $\{X_t\}_{t \in \N}$ with Metropolis transition probabilities, parametrized by the fugacity $\l \geq 1$. More precisely, at every step a site is selected uniformly at random; if such a site is unoccupied, then a particle is placed there with probability $1$ if and only if all the neighboring sites are also unoccupied; if instead the selected site is occupied, the particle is removed with probability $1/\l$.

Other single-site update dynamics (\eg Glauber dynamics) for the hard-core model have received a lot of attention in the discrete mathematics community~\cite{BCFKTVV99,Dyer2002,G08,GT04,GT06,GR10,LV99,MWW08}, where they are instrumental to sample weighted independent sets. Aiming to understand the performance of this local Markov chain Monte Carlo method, the main focus of this literature is on \textit{mixing times} and on how they scale in the graph size. Indeed, the behavior of these MCMC changes dramatically as the fugacity $\l$ grows, going from a fast convergence to stationarity (``\textit{fast mixing}'') to an exponentially slow one (``\textit{slow mixing}''); this phenomenon is intimately related to the aforementioned phase transition phenomenon of the hard-core model on infinite graphs.

The main focus of the present paper is on the hard-core particle dynamics $\{X_t\}_{t \in \N}$ on finite graphs when the fugacity grows large, \ie$\l \to \infty$. In this regime, the hard-core measure~\eqref{eq:hcgb} favors configurations with a maximum number of particles and we are interested in describing the \textit{tunneling behavior} of such a particle system, \ie how it evolves between these maximum-occupancy configurations. To understand the transient behavior of the hard-core model in the high-fugacity regime, we study the asymptotic behavior of the \textit{first hitting times} of the Markov chain $\{X_t\}_{t \in \N}$ between the maximum-occupancy configurations, which tell us how ``rigid'' they are and how long it takes for the particle system to ``switch'' between them.

The hard-core model has been successfully used to model certain random-access protocols for wireless networks~\cite{DDT08,WK05,Y83}. In this context understanding the tunneling behavior of the hard-core model is instrumental to analyze \textit{temporal starvation phenomena} for these communication networks and their impact on performance~\cite{ZBvLN13}.

Tunneling phenomena of the hard-core model have already been studied on complete partite graphs~\cite{ZBvL12,ZBvL15} and on square grid graphs~\cite{NZB16}. In this work we focus on the case where $\L$ is a \textit{finite triangular lattice}. Imposing periodic boundary conditions, there are three maximum-occupancy configurations on such graphs, as illustrated in Figure~\ref{fig:threedominantstates}. These three hard-core configurations, denoted as $\aa$, $\bb$, and $\cc$, correspond to the tripartition of $\L$. \vspace{-0.1cm}
{\captionsetup[subfigure]{labelformat=empty}%
\begin{figure}[!ht]
	\centering
	\subfloat{\includegraphics[scale=0.99]{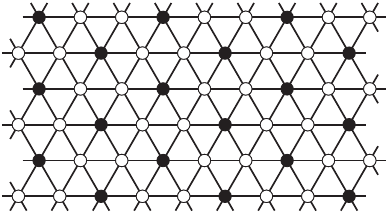}}
	\hspace{1.5cm}
	\subfloat{\includegraphics[scale=0.99]{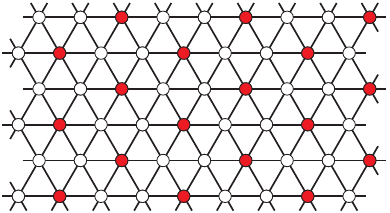}}
	\vspace{0.3cm}\\
	\subfloat{\includegraphics[scale=0.99]{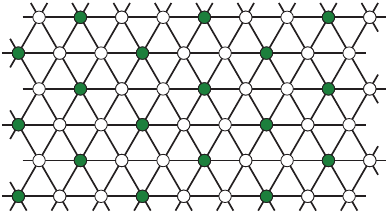}}
	\caption{The three maximum-occupancy configurations $\aa$, $\bb$, and $\cc$ on the $6 \times 9$ triangular lattice} 
	\label{fig:threedominantstates}
\end{figure}}
\FloatBarrier
As the fugacity grows large, this particle system spends roughly one third of the time in each of these three configurations. However, it takes a long time for the Markov chain $\{X_t\}_{t \in \N}$ to move from one maximum-occupancy configuration to another, since such a transition involves the occurrence of rare events. Intuitively, along any such a transition, the Markov chain must follow a path through mixed-activity particle patterns that, having fewer particles, are highly unlikely in view of~\eqref{eq:hcgb} and the time to reach such configurations is correspondingly long. 

By introducing the \textit{inverse temperature} $\b = \log \l$ and an appropriate Hamiltonian, the Markov chain $\{X_t\}_{t\in \N}$ can be seen as a Freidlin-Wentzell Markov chain with Metropolis transition probabilities and the hard-core measure~\eqref{eq:hcgb} rewrites as a Gibbs distribution. In this setting, the high-fugacity regime in which we are interested corresponds to the low-temperature limit $\binf$ and the maximum-occupancy configurations are the \textit{stable configurations} of the system, \ie the global minima of the Hamiltonian.

This identification allow us to use the \textit{pathwise approach}~\cite{CGOV84,MNOS04}, a framework that has been successfully used to study metastability problem for many finite-volume models in a low-temperature regime. 
In this paper we mostly make use of the extension of the classical pathwise approach developed in~\cite{NZB16} that covers the case of \textit{tunneling times}. The crucial idea behind this method is to understand which paths the Markov chain most likely follows in the low-temperature regime and derive from them asymptotic results for hitting times. For Freidlin-Wentzell Markov chains this can be done by analyzing the \textit{energy landscape} to find the paths between the initial and the target configurations with a minimum energy barrier. In the case of the tunneling times between stable configurations of the hard-core model, this problem reduces to identifying the most efficient way, starting from a stable configuration to progressively add the particles present in the target stable configuration.



By exploring detailed geometric properties of the mixed-activity hard-core configurations on finite triangular lattices, we develop a novel combinatorial method to quantify their ``energy inefficiency'' and obtain in this way the \textit{minimum energy barrier} $\G(\L)>0$ that has to be overcome in the energy landscape for the required transition to occur. In particular, we show how this minimum energy barrier $\G(\L)$ depends on the sizes of the finite triangular lattice $\L$. In our main result we characterize the asymptotic behavior for the tunneling times between stable configurations giving sharp bounds in probability and proving that the order of magnitude of their expected values is equal to $\G(\L)$ on a logarithmic scale.
Furthermore, we prove that the tunneling times scaled by their expected values are exponentially distributed in the low-temperature limit, leveraging in a nontrivial way the intrinsic symmetry of the energy landscape. 

Lastly, using structural properties of the energy landscapes and classical results~\cite{C99,Miclo2002} for Freidlin-Wentzell Markov chains, we show that the timescale $\l^{\G(\L)} = e^{\b \G(\L)}$ at which transitions between maximum-occupancy configurations most likely occur is also the order of magnitude of the \textit{mixing time} of the Markov chain $\{X_t\}_{t\in \N}$, proving that the hard-core dynamics exhibit \textit{slow mixing} on finite triangular lattices.

\section{Model description and main results}
\label{sec2}
We consider the hard-core model on finite triangular lattices with periodic boundary conditions. More precisely, given two integers $\vdtr \geq 2$ and $\hdtr \geq 1$, we consider the $2K \times 3L$ \textit{triangular grid} $\L$, that is the subgraph of the triangular lattice consisting of $N := |V| = 6 \vdtr \hdtr$ sites placed on $2 \vdtr$ rows of $3 \hdtr$ sites each
, see Figure~\ref{fig:Tmnillustration}. 
\begin{figure}[!ht]
	\centering
	\includegraphics{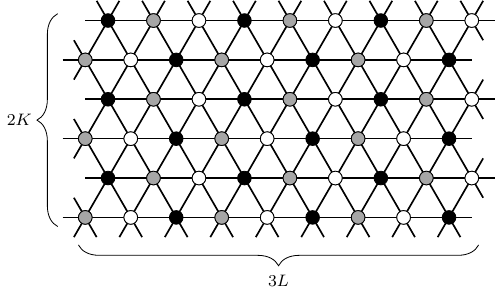}
	\caption{The $6 \times 9$ triangular grid $\L$ and its three components highlighted using different colors}
	\label{fig:Tmnillustration}
\end{figure}
\FloatBarrier
We impose periodic boundary conditions on $\L$ to preserve symmetry: indeed this choice makes $\L$ a vertex-transitive graph in which every vertex has the same local neighborhood. The triangular grid $\L$ has a natural tri-partition $V=\AA \cup \BB \cup \CC$, which is highlighted in Figure~\ref{fig:Tmnillustration} by coloring the three components in gray, black, and white, respectively. Thanks to the chosen sizes, the three components of $\L$ have the same number of sites
\begin{equation}
\label{eq:componentsizestriangulargrid}
	|\AA|=|\BB|=|\CC| = \frac{N}{3} = 2 \vdtr \hdtr.
\end{equation}
A particle configuration on $\L$ is a map $\s: V \to \{0,1\}$, in which we set $\s(v)=1$ when the site $v$ is occupied and $\s(v)=0$ otherwise.
A particle configuration on $\L$ is a \textit{hard-core configuration} if $\s(v)\s(w) = 0$ for every pair of neighboring sites $v,w$. We denote by $\cX \subset \{0,1\}^N$ the set of all hard-core configurations on $\L$.

Let $\aa,\bb$ and $\cc$ be the hard-core configurations on the triangular grid $\L$ defined as
\[
	\aa(v):=\mathds{1}_{\{v \in \AA\}}(v),
	\quad
	\bb(v):=\mathds{1}_{\{v \in \BB\}}(v),
	\quad \text{ and } \quad
	\cc(v):=\mathds{1}_{\{v \in \CC\}}(v).
\]
In Section~\ref{sec3} we show that $\aa,\bb$ and $\cc$ are the maximum-occupancy configurations of the hard-core model on $\L$.

We are interested in studying the Metropolis dynamics for such a model, that is the family of Markov chains $\{X^\b_t\}_{t \in \N}$ on $\cX$ parametrized by the inverse temperature $\b >0$ with transition probabilities
\[
	P_\b(\s,\s'):=
	\begin{cases}
		Q(\s,\s') e^{-\b [H(\s')-H(\s)]^+}, 	& \text{ if } \s \neq \s',\\
		1-\sum_{\h \neq \s } P_\b(\s,\h), 	& \text{ if } \s=\s',
	\end{cases}
\]
where the \textit{connectivity function} $Q: \{ (\s,\s') \in \cX \times \cX ~:~ \s\neq \s'\} \to [0,1]$ allows only single-site updates:
\begin{equation}
\label{eq:connectivityfunction_TR}
	Q(\s,\s'):=
	\begin{cases}
		\frac{1}{N}, 									& \text{ if } |\{v \in V ~:~ \s(v) \neq \s'(v)\}|=1,\\
		0, 													& \text{ if } |\{v \in V ~:~ \s(v) \neq \s'(v)\}| > 1,
	\end{cases}
\end{equation}
and $H: \cX \to \R$  is the \textit{energy function} or \textit{Hamiltonian} defined as
\begin{equation}
\label{eq:energyfunction_TR}
	H(\s):= -\sum_{v \in V} \s(v).
\end{equation}
In other words, each configuration $\s \in \cX$ is assigned an \textit{energy} $H(\s)$ proportional to the total number of particles in $\s$. We remark that here the energy of a hard-core configuration does not describe the interaction potential between particles, which is already fully captured by the set $\cX$ of hard-core configurations on $\L$.

The triplet $(\cX,H,Q)$ is called \textit{energy landscape} and we denote by $\ss$ the set of \textit{stable configurations} of the energy landscape, that is the set of global minima of $H$ on $\cX$. Since $\cX$ is finite, the set $\ss$ is always nonempty. 

The Markov chain $\smash{\xtbb}$ is reversible with respect to the \textit{Gibbs measure} at inverse temperature $\b$ associated to the Hamiltonian $H$, namely
\[
	\mu_\b(\s):=\frac{1}{Z_\b} e^{-\b H(\s)}, \quad \s \in \cX,
\]
where $Z_\b:=\sum_{\s' \in \cX} e^{-\b H(\s')}$ is the normalizing partition function. Furthermore, it is well-known (\eg \cite[Proposition 1.1]{C99}) that the Markov chain $\smash{\xtbb}$ is aperiodic and irreducible on $\cX$. Hence, $\smash{\xtbb}$ is ergodic on $\cX$ with stationary distribution $\mu_\b$. For a nonempty subset $A \subset \cX$ and a configuration $\s \in \cX$, we denote by $\tha$ the \textit{first hitting time} of the subset $A$ for the Markov chain $\smash{\xtbb}$ with initial configuration $\s$ at time $t=0$, \ie
\[
	\tha:=\inf \{ t\in \N  : X^{\b}_t \in A \st X^{\b}_0=\s\}.
\]
We will refer to $\tha$ as \textit{tunneling time} if $\s$ is a stable configuration and the target set is some $A \subseteq \ss \setminus \{\s\}$.

The first main result describes the asymptotic behavior of the tunneling times $\t^{\aa}_{\bb}$ and $\t^{\aa}_{\{\bb,\cc\}}$ on the triangular grid $\L$ in the low-temperature regime $\binf$. 

\begin{thm}[Asymptotic behavior of tunneling times]
\label{thm:tabc}
Consider the Metropolis Markov chain $\{ X^\b_t \}_{t \in \N}$ corresponding to the hard-core dynamics on the $2K \times 3L$ triangular grid $\L$ and define
\begin{equation}
\label{def:gamma_tr}
	\G(\L):=\min\{ \vdtr, 2\hdtr \} +1.
\end{equation}
Then,
\begin{itemize}
	\item[\textup{(i)}] $\displaystyle \limb \mathbb P_\b  \Bigl ( e^{\b (\G(\L) - \e )} \leq \t^{\aa}_{\bb} \leq \t^{\aa}_{\{\bb,\cc\}} \leq e^{\b (\G(\L) + \e )} \Bigr ) =1;$
	\item[\textup{(ii)}] $\displaystyle \limb \frac{1}{\b} \log \E \t^{\aa}_{\bb} = \G(\L) = \limb \frac{1}{\b} \log \E \t^{\aa}_{\{\bb,\cc\}}; $
	\item[\textup{(iii)}] $\displaystyle \frac{\t^{\aa}_{\{\bb,\cc\}}}{\E \t^{\aa}_{\{\bb,\cc\}}} \cd \mathrm{Exp}(1), \quad \mathrm{ as } \, \, \binf;$
	\item[\textup{(iv)}] $\displaystyle \frac{\t^{\aa}_{\bb}}{\E \t^{\aa}_{\bb}} \cd \mathrm{Exp}(1), \quad \mathrm{ as } \, \, \binf.$
\end{itemize}
\end{thm}

The proofs of statements (i), (ii), and (iii) of the latter theorem are presented in Section~\ref{newsec4} and leverage the general framework for hitting time asymptotics developed in~\cite{NZB16} in combination with the analysis of the energy landscape corresponding to the hard-core model on the triangular grid $\L$ to which Section~\ref{sec3} is devoted.

As established by the next theorem, which is our second main result, the structural properties of the energy landscapes that will be presented in Section~\ref{sec3} also yield the following result for the \textit{mixing time}. Recall that the mixing time describes the time required for the distance (measured in total variation) to stationarity to become small. More precisely, for every $0 < \e < 1$, we define the \textit{mixing time} $t^{\mathrm{mix}}_\b(\e)$ by
\[
	t^{\mathrm{mix}}_\b(\e):=\min\{ n \geq 0 ~:~ \max_{\s \in \cX} \| P^n_\b(\s,\cdot) - \mu_\b(\cdot) \|_{\mathrm{TV}} \leq \e \},
\]
where $\| \nu - \nu' \|_{\mathrm{TV}}:=\frac{1}{2} \sum_{\s \in \cX} |\nu(\s)-\nu'(\s)|$ for any two probability distributions $\nu,\nu'$ on $\cX$. Another classical notion to investigate the speed of convergence of Markov chains is the \textit{spectral gap}, which is defined as $	\smash{\rho_\b := 1-\a_2}$, where $\smash{1=\a_{1} > \a_2 \geq \dots \geq \a_{|\cX|} \geq -1}$ are the eigenvalues of the matrix $\smash{(P_\b(\s,\s'))_{\s,\s' \in \cX}}$. The spectral gap can be equivalently defined using the Dirichlet form associated with the pair $(P_\b, \mu_\b)$, see~\cite[Lemma 13.12]{LPW09}.

\begin{thm}[Mixing time and spectral gap]\label{thm:mix_triangularlattice}
Consider the Metropolis Markov chain $\smash{\xtbb}$ corresponding to the hard-core dynamics on the $2K \times 3L$ triangular grid $\L$ and define $\G(\L)$ as in~\eqref{def:gamma_tr}. Then, for any $0 < \e < 1$, 
\[
	\limb \frac{1}{\b} \log t^{\mathrm{mix}}_\b(\e) = \G(\L).
\]
Furthermore, there exist two positive constants $0 < c_1 \leq c_2 < \infty$ independent of $\b$ such that the spectral gap $\rho_\b$ of the Markov chain $\smash{\xtbb}$ satisfies
\[
	c_1 e^{-\b \G(\L)} \leq \rho_\b \leq c_2 e^{-\b \G(\L)} \qquad \forall \, \b \geq 0.
\]
\end{thm}

Therefore, the mixing time turns out to be asymptotically of the same order of magnitude as the tunneling time between stable configurations, establishing the \textit{slow mixing} of $\smash{\xtbb}$ as $\binf$. We remark that mixing times for the hard-core model with Glauber dynamics have received a lot of attention in the literature, see, e.g., \cite{BCFKTVV99,GT04,GT06} and, in particular, \cite{GR07} for results for triangular grids, but, differently from the present paper, these works focus mostly on identifying how the mixing time scales with the graph size at fixed temperature/fugacity.

Different choices for the sizes of the triangular grid or for the boundary conditions result in a fundamentally different geometry of the stable configurations and thus completely change the energy landscape. For example, (i) the $4 \times 6$ triangular grid with open boundary conditions is still a tripartite graph, but has $63$ stable configurations; (ii) the $4\times 4$ and $4\times 5$ triangular grids with periodic boundary conditions are not tripartite anymore (being both $4$-partite graphs) and have $32$ and $10$ stable configurations, respectively; (iii) the $5\times 5$ triangular grid with open boundary conditions is $4$-partite and has a unique stable configuration.
A complete characterization of the stable configurations of these triangular grids and of the (probably heterogeneous) energy barriers separating them seems very involved and is only the first step of the energy landscape analysis, leaving little hope that our results could be easily generalized to such scenarios. 

The rest of the paper is organized as follows. Section~\ref{sec3} is entirely devoted to analysis of geometrical and combinatorial properties of the hard-core configurations on triangular grids and to the derivation of the structural properties of the energy landscape, which will then be used in Section~\ref{newsec4} to prove the two main theorems.

\section{Energy landscape analysis}
\label{sec3}
This section is devoted to the analysis of the energy landscape associated with the hard-core dynamics on the $2K \times 3L$ triangular grid $\L$. Leveraging geometrical features of the hard-core configurations on $\L$, we prove crucial structural properties of the corresponding energy landscape $(\cX,H,Q)$, as stated in Theorem~\ref{thm:structuralproperties} below.

In the rest of this paper, we will use the same notions and notation introduced in~\cite{NZB16}. The connectivity matrix $Q$ given in~\eqref{eq:connectivityfunction_TR} is irreducible, \ie for any pair of configurations $\s,\s'\in \cX$, $\s \neq \s'$, there exists a finite sequence $\o$ of configurations $\o_1,\dots,\o_n \in \cX$ such that $\o_1=\s$, $\o_n=\s'$ and $Q(\o_i,\o_{i+1})>0$, for $i=1,\dots, n-1$. We refer to such a sequence as a \textit{path} from $\s$ to $\s'$ and denote it by $\o: \s \to \s'$. Given a path $\o=(\o_1,\dots,\o_n)$, we define its \textit{height} $\Phi_\o$ as
$
	\Phi_\o:= \max_{i=1,\dots,n} H(\o_i).
$
The \textit{communication height} between a pair of configurations $\s,\s' \in\cX$ is defined as
\[
	\Phi(\s,\s') := \min_{\o : \s\to \s'} \Phi_\o = \min_{\o : \s\to \s'} \max_{i=1,\dots,|\o|} H(\o_i),
\]
and its natural extension to disjoint non-empty subsets $A, B \subset \cX$ is
\[
	\Phi(A,B) := \min_{\s \in A, \, \s' \in B} \Phi(\s,\s').
\]

The next theorem summarizes the structural properties of the energy landscape corresponding to the hard-core dynamics on a triangular grid $\L$. More specifically, (i) we prove that $\aa$, $\bb$ and $\cc$ are the only three stable configurations, (ii) find the value of the communication height between them, as a function of the triangular grid sizes $\vdtr$ and $\hdtr$, and (iii) show by means of two iterative algorithms that there is ``\textit{absence of deep cycles}'' (see condition~\eqref{eq:suffcondPE} below) in the energy landscape $(\cX,H,Q)$.

\begin{thm}[Structural properties of the energy landscape]\label{thm:structuralproperties}
Let $(\cX,H,Q)$ be the energy landscape corresponding to the hard-core dynamics on the $2K \times 3L$ triangular grid $\L$ with $\vdtr \geq 2$ and $\hdtr \geq 1$. Then,
\begin{itemize}
	\item[\textup{(i)}] $ \displaystyle \ss = \{ \aa, \bb, \cc\}$;
	\item[\textup{(ii)}] $\Phi(\aa,\bb) - H(\aa)=\Phi(\aa,\cc) - H(\aa) =\Phi(\bb,\cc) - H(\bb)= \min\{\vdtr, 2\hdtr\}+1$;
	\item[\textup{(iii)}] $\Phi(\s,\{\aa,\bb,\cc\}) - H(\s) \leq \min\{\vdtr, 2\hdtr\} \quad \forall \, \s \in \cX \setminus \{ \aa,\bb,\cc \}$.
\end{itemize}
\end{thm}
Note that identity (ii) in Theorem~\ref{thm:structuralproperties} motivates the definition~\eqref{def:gamma_tr} of $\G(\L)$ in Theorem~\ref{thm:tabc}. 

We briefly outline here the proof strategy of Theorem~\ref{thm:structuralproperties}, to which is devoted the rest of this section. As illustrated by the state space diagram in Figure~\ref{fig:ss4x6} below, there is not a unique bottleneck separating the stable configurations (this was the case for complete partite graphs~\cite{ZBvL12,ZBvL15}) and there are in fact exponentially many possible ways for the Markov chain $\smash{\xtbb}$ to make such transitions. This makes the task of identifying the energy barrier between stable configurations much harder. 

Inspired by the ideas in~\cite{GR10} and by the methodology used for square grids in~\cite{NZB16} we tackle this problem by looking at geometric features of the hard-core configurations on triangular grids. In Subsection~\ref{sub51}, after some preliminary definitions, we study the combinatorial properties of hard-core configurations on horizontal and vertical \textit{stripes} of the triangular grid $\L$, \ie pairs of adjacent rows (triplets of adjacent columns, respectively). In particular, we find the maximum number of particles that a hard-core configuration can have in a horizontal stripe and characterize how particles are arranged on such stripes in Lemma~\ref{lem:efficientstripes}. Theorem~\ref{thm:structuralproperties}(i) is an almost immediate consequence of these combinatorial results. Afterwards, using further geometrical properties of the hard-core configurations, we prove Proposition~\ref{prop:lowerphiaabb}, which gives the following lower bound for the communication height between $\aa$ and $\bb$:
\[
	\Phi(\aa,\bb) - H(\aa) \geq \min\{\vdtr, 2\hdtr\} +1.
\]
We then introduce two \textit{energy reduction algorithms} in Subsection~\ref{sub52}, which are used in Proposition~\ref{prop:refpathaatobb} to construct a reference path $\o^*: \aa \to \bb$, guaranteeing that the lower bound above is sharp, i.e.,
\[
	\Phi(\aa,\bb) - H(\aa) = \min\{\vdtr, 2\hdtr\} +1,
\]
and concluding the proof of Theorem~\ref{thm:structuralproperties}(ii). The same algorithms are then used again to build a path from every configuration $\s \not\in\{ \aa, \bb, \cc\}$ to the set $\{ \aa, \bb, \cc \}$ with a prescribed energy height, obtaining in this way the inequality stated in Theorem~\ref{thm:structuralproperties}(iii). 
\begin{figure}[!ht]
	\centering
	\includegraphics[scale=0.3]{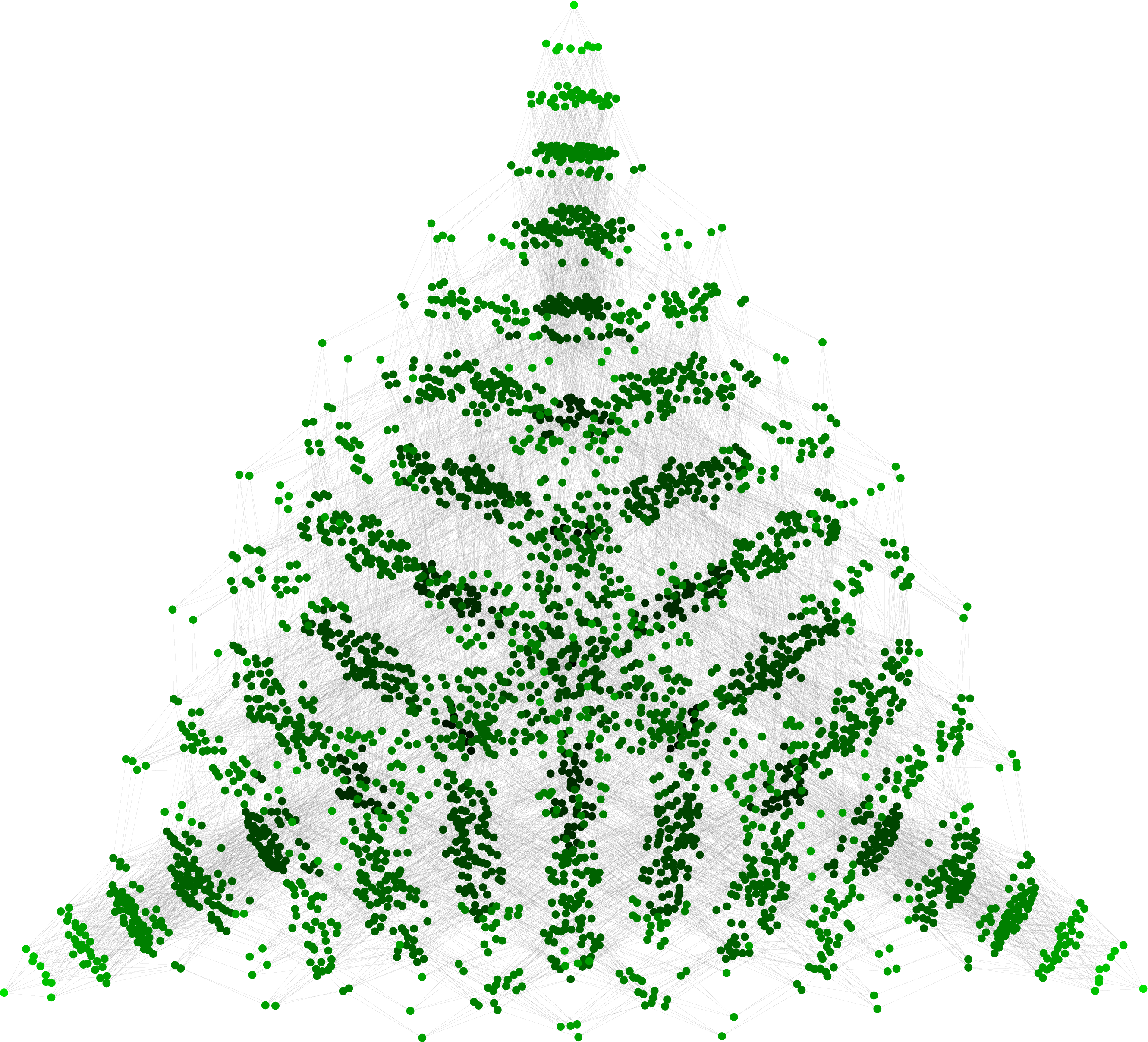}
	\caption{The energy landscape corresponding to the hard-core dynamics on the $4 \times 6$ triangular grid. The color scheme is chosen in such a way that the lighter the color of a node, the lower the energy of the corresponding configuration.}
	\label{fig:ss4x6}
\end{figure}
\FloatBarrier


\subsection{Geometrical properties of hard-core configurations}
\label{sub51}
We first introduce some useful definitions to describe hard-core configurations on the triangular grid $\L$. Denote by $c_j$, $j=0,\dots, 6 \hdtr -1$, the $j$-th column of $\L$, and by $r_i$, $i=0,\dots,2 \vdtr-1$, the $i$-th row of $\L$, see Figure~\ref{fig:stripestriangulargrid}. In the rest of the paper, the row and column indices should always be taken modulo $2K$ and $6L$, respectively.

\begin{figure}[!ht]
	\centering
	\includegraphics{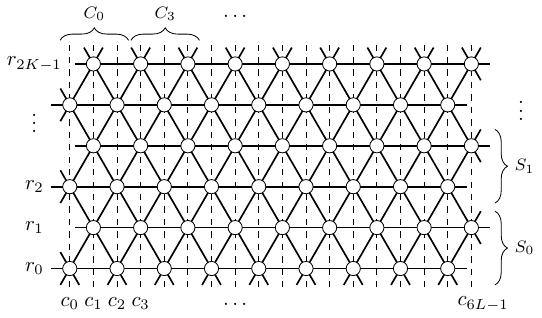}
	\caption{Illustration of row, column and stripe notation for the triangular grid}
	\label{fig:stripestriangulargrid}
\end{figure}

Note that every row has an equal number of sites from each component, since
\begin{equation}
\label{eq:rowsarebalanced}
	|r_i \cap \AA| = |r_i \cap \BB| = |r_i \cap \CC| = \hdtr \qquad \forall \, i=0,\dots,2 \vdtr-1,
\end{equation}
while each column consists of sites from a single component, and, in fact,
\begin{equation}
\label{eq:columnsaremonochromo}
	\AA = \bigcup_{j=0}^{\hdtr-1} c_{3j}, \quad \BB = \bigcup_{j=0}^{\hdtr-1} c_{3j+1}, \quad \text{ and } \quad \CC = \bigcup_{j=0}^{\hdtr-1} c_{3j+2}.
\end{equation}
Each site $v \in V$ lies at the intersection of a row with a column and we associate to $v$ the coordinates $(i,j)$ if $v = r_j \cap c_i$. We call the collection of sites belonging to two adjacent rows a \textit{horizontal stripe}. In particular, we denote by $S_i$, with $i=0,\dots,\vdtr-1$, the horizontal stripe consisting of rows $r_{2i}$ and $r_{2i+1}$, \ie $S_i:=r_{2i} \cup r_{2i+1}$, see Figure~\ref{fig:stripestriangulargrid}. When the index of a stripe is not relevant, we will simply denote it by $S$. We define a \textit{vertical stripe} to be the collection of sites belonging to three adjacent columns, which we denote by $C$ in general. In particular, for $j=0,\dots,3\hdtr-1$ we denote by $C_{j}$ the vertical stripe consisting of columns $c_{j}$, $c_{j+1}$ and $c_{j+2}$, see Figure~\ref{fig:stripestriangulargrid}. For every horizontal stripe $S$ note that $|S|=6 \hdtr$ and~\eqref{eq:rowsarebalanced} implies that 
$
	|S \cap \AA| = |S \cap \BB| = |S \cap \CC| = 2 \hdtr,
$
see also Figure~\ref{fig:Si} where we highlight the tripartition of a horizontal stripe. Similarly, for every vertical stripe $C$, we have $|C|=3 \vdtr$ and, in view of~\eqref{eq:columnsaremonochromo}, we have 
$
	|C \cap \AA| = |C \cap \BB| = |C \cap \CC| = \vdtr.
$
A special role will be played by the vertical stripes whose middle column belongs to $\BB$, which are those of the form $C_{3j}$ for some $j=0,\dots,2\hdtr-1$, whose structure is displayed in Figures~\ref{fig:C3jeven} and~\ref{fig:C3jodd}.

\begin{figure}[!ht]
	\centering
	\subfloat[A horizontal stripe $S_i$ of the $2\vdtr \times 9$ triangular grid\label{fig:Si}]{
	\includegraphics{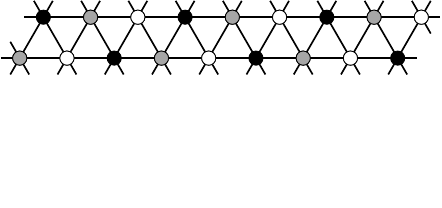}}
	\hspace{0.2cm}
	\subfloat[A vertical stripe $C_{3j}$ for $j$ even of a $8 \times 3 \hdtr$ triangular grid\label{fig:C3jeven}]{
	\includegraphics{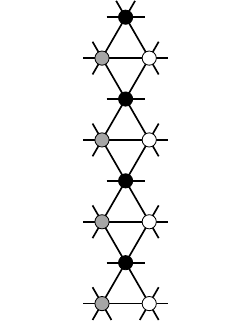}}
	\hspace{0.2cm}
	\subfloat[A vertical stripe $C_{3j}$ for $j$ odd of a $8 \times 3 \hdtr$ triangular grid\label{fig:C3jodd}]{
	\includegraphics{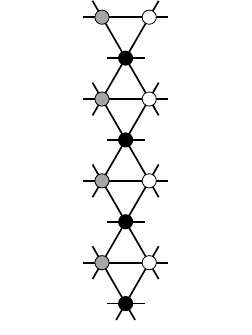}}
	\caption{Illustration of horizontal and vertical stripes in which the sites' tripartition is highlighted using different colors}
	\label{fig:verticalhorizontalstripes}
\end{figure}
Given a hard-core configuration $\s \in \cX$, we define its \textit{energy difference} $\DH(\s)$ as 
\begin{equation}
\label{def:energywastagetr}
	\DH(\s) := H(\s) - H(\aa).
\end{equation}
In view of the fact that $H(\aa)=H(\bb)=H(\cc) = - 2 \vdtr \hdtr$ and the  definition~\eqref{eq:energyfunction_TR} of $H(\cdot)$, we can rewrite
\[
	\DH(\s) = 2 \vdtr \hdtr - \sum_{v \in V} \s(v).
\]
A subset of sites $W \subseteq V$ is said to be \textit{balanced} if $|W \cap \L_\aa| =|W \cap \L_\bb| =|W \cap \L_\cc| $. The energy difference of a configuration $\s \in \cX$ on a balanced subset $W \subseteq V$ is defined as by $\Delta H_W(\s) := |W \cap \L_\aa| - \sum_{v \in W} \s(v)$.
Note that the horizontal and vertical stripes are balanced subsets and that energy difference $\DH(\s)$ in~\eqref{def:energywastagetr} can be written as the sum of the energy differences on non-overlapping horizontal/vertical stripes, \ie
\begin{equation}
\label{eq:energywastagesumtr}
	\DH(\s) = \sum_{i=0}^{\vdtr-1} \DH_{S_i}(\s) = \sum_{j=0}^{2 \hdtr-1} \DH_{C_{3j}}(\s) = \sum_{j=0}^{2 \hdtr-1} \DH_{C_{3j+1}}(\s) = \sum_{j=0}^{2 \hdtr-1} \DH_{C_{3j+2}}(\s).
\end{equation}

We adopt the following coloring convention for displaying a hard-core configuration $\s \in \cX$: We put a node in site $v\in V$ if it is occupied, \ie $\s(v)=1$, and we color it gray, black, or white depending on whether the site $v$ belongs to $\AA$, $\BB$, $\CC$ respectively; if a site $v\in V$ is unoccupied, \ie $\s(v)=0$, we do not display any node there.

There is an equivalent way to represent hard-core configurations on $\L$. Consider the $12 K L$ triangular faces of the graph $\L$, to which we will simply refer as \textit{triangles}. Each triangle can have at most one occupied site in its three vertices (them being a clique of the graph $\L$); if this is the case, then we refer to it as \textit{blocked triangle} and color it as gray, black, or with a dotted pattern, depending on whether such particle belongs to $\AA$, $\BB$ or $\CC$, respectively. Otherwise, if none of its three vertices is occupied by a particle, we call a triangle \textit{free} and leave it blank. In the rest of the paper, we will use a ``mixed'' representation for hard-core configurations on $\L$, displaying both the occupied sites and the corresponding blocked triangles with the aforementioned coloring schemes, see~Figure~\ref{fig:example_new} for an example.
\begin{figure}[!ht]
	\centering
	\includegraphics{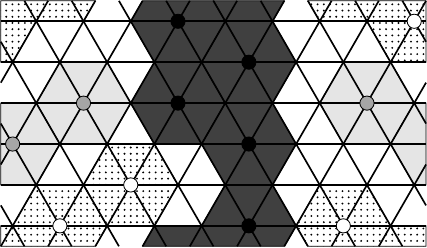}
	\caption{An example of a hard-core configuration $\s$ on the $6 \times 9$ triangular grid}
	\label{fig:example_new}
\end{figure}
Since each site is the vertex of six triangles on $\L$, placing particles with hard-core constraints on a triangular grid corresponds to placing hexagons without overlaps on the same lattice. This is the reason why the hard-core model on the triangular lattice is often called \textit{hard-hexagon model} in the statistical physics literature. 

\begin{rem} A key observation is that blocked triangles sharing an edge must be of the same color. Indeed, as illustrated in Figure~\ref{fig:sharingedges}, if a particle resides in one of the two endpoints of that edge, they trivially are of the same color by construction. Otherwise, there must be a particle in each of the two vertices that are not shared by the two triangles (as they are both assumed to be blocked). It is easy to check that these two vertices always belong to the same partition (cf.~Figure~\ref{fig:sharingedges}), yielding the same coloring for the two triangles under consideration. \vspace{-0.7cm}
\end{rem}
\begin{figure}[!ht]
	\centering
	\subfloat{
	\includegraphics{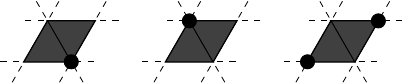}}
	\hspace{1cm}
	\subfloat{
	\includegraphics{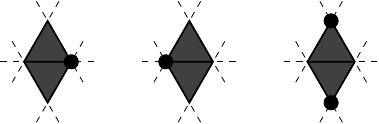}}
	\caption{All possible local hard-core configurations in which two blocked triangles share an edge}
	\label{fig:sharingedges}
\end{figure}
\FloatBarrier
Given two configurations $\s,\s' \in \cX$ and a subset of sites $W \subseteq V$, we write
\[
	\s_{|W} = \s'_{|W} \quad \Longleftrightarrow \quad \s(v) = \s'(v) \quad \forall \, v \in W.
\]
We say that a configuration $\s \in \cX$ has a \textit{horizontal} $\aa$--($\bb$--,$\cc$--)\textit{bridge} in stripe $S$ if $\s$ perfectly agrees there with $\aa$ (respectively $\bb$, $\cc$), \ie
$
	\s_{|S} = \aa_{|S} \, (\text{respectively } \s_{|S} = \bb_{|S} \text{ or } \s_{|S} = \cc_{|S}).
$
Similarly, we say that $\s \in \cX$ has a \textit{vertical} $\aa$--($\bb$--,$\cc$--)\textit{bridge} in stripe $C$ if $\s$ perfectly agrees there with $\aa$ (respectively $\bb$, $\cc$), \ie
$
	\s_{|C} = \aa_{|C} \, (\text{respectively } \s_{|C} = \bb_{|C} \text{ or } \s_{|C} = \cc_{|C}).
$
We informally say that two bridges are of the same \textit{color} when they agree with the same stable configuration. Two examples of bridges are shown in Figures~\ref{fig:blackvertbridge} and~\ref{fig:grayhoribridge}.

\begin{figure}[!ht]
	\centering
	\subfloat[Vertical $\bb$--bridge in $C_9$\label{fig:blackvertbridge}]{
	\includegraphics{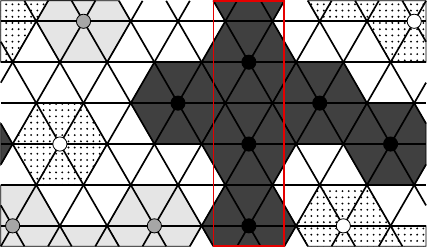}}
	\hspace{0.5cm}
	\subfloat[Horizontal $\aa$--bridge in $S_1$\label{fig:grayhoribridge}]{
	\includegraphics{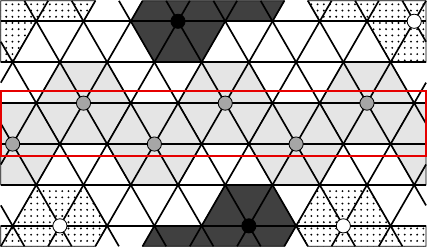}}
	\caption{Examples of hard-core configurations displaying bridges on the $6 \times 9$ triangular grid}
\end{figure}
\FloatBarrier

\begin{lem}[Geometric features of hard-core configurations]\label{lem:crossesaremonochromatic}
A hard-core configuration $\s \in \cX$ cannot display simultaneously a vertical bridge and a horizontal bridge of different colors.
\end{lem}
\begin{proof}
Assume without loss of generality that the vertical bridge is a $\bb$--bridge. Such a vertical bridge blocks two sites on every row, belonging to $\AA$ and $\CC$, and thus no horizontal stripe can fully agree with $\aa$ or $\cc$.
\end{proof}
It is possible, however, that a vertical and a horizontal bridges coexist when they are of the same color and this fact motivates the next definition. We say that a configuration $\s \in \cX$ has a $\aa$--($\bb$--,$\cc$--)\textit{cross} if it has simultaneously at least two $\aa$--($\bb$--,$\cc$--)bridges, one vertical and one horizontal; see Figure~\ref{fig:crosstr} for an example of a $\bb$--cross.


\begin{figure}[!ht]
	\centering
	\includegraphics{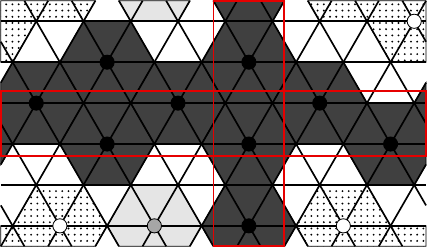}
	\caption{Example of a hard-core configuration displaying a $\bb$--cross on the $6 \times 9$ triangular grid}
	\label{fig:crosstr}
\end{figure}
\FloatBarrier

In order to prove Theorem~\ref{thm:structuralproperties}(ii), we need the following lemma which characterizes the structure of horizontal and vertical stripes with zero energy difference.

\begin{lem}[Energy-efficient stripes structure]\label{lem:efficientstripes}
Let $\s \in \cX$  be a hard-core configuration on the $2K \times 3L$ triangular grid. The following statements hold:
\begin{itemize}
\item[\textup{(i)}] For every horizontal stripe $S$, the energy difference is non-negative, \ie $\DH_{S}(\s) \geq 0$, and
\begin{equation}
\label{eq:efficienthorizontalstripes}
	\DH_{S}(\s)=0 \quad \Longleftrightarrow \quad \s \text{ has a horizontal bridge in stripe } S;
\end{equation}
\item[\textup{(ii)}] For every vertical stripe $C$ of the form $C=C_{3j}$, the energy difference is non-negative, \ie $\DH_{C}(\s) \geq 0$. Furthermore, if $\s$ has at least one black particle on $C$, i.e., $\sum_{v \in C \cap \BB} \s(v) >0$, then
\begin{equation}
\label{eq:verticalstripeenergywastage}
	\DH_{C}(\s)=0 \quad \Longleftrightarrow \quad \s \text{ has a vertical $\bb$--bridge in stripe } C.
\end{equation}
\end{itemize}
\end{lem}
\begin{proof}
In this proof we leverage the equivalent representation of a hard-core configuration as collection of blocked triangles. The underlying idea for horizontal and vertical stripes is the same, but we present the proof separately in view of their different structures.

(i) Given $\s \in \cX$, denote by $b_S(\s)$ and $f_S(\s)$ the number of blocked triangles and of free triangles on the horizontal stripe $S$, respectively. Since the total number of triangles of the horizontal stripe $S$ is $6 \hdtr$, we have $b_S(\s) + f_S(\s) = 6 \hdtr$. Furthermore, as each particle blocks exactly $3$ triangles on the horizontal stripe $S$, it holds that $b_S(\s)=3 \sum_{v \in S} \s(v)$ and, thus,
\begin{equation}
\label{eq:freetriangleUS}
	f_S(\s) = 6 \hdtr - b_S(\s) = 3 \Big ( 2 \hdtr - \sum_{v \in S} \s(v) \Big ) = 3 \cdot \DH_S(\s).
\end{equation}
Since $f_S(\s)$ is by construction a non-negative integer, it readily follows that  $\DH_S(\s) \geq 0$.

Let us now turn to the characterization~\eqref{eq:efficienthorizontalstripes} of the horizontal stripes with energy difference equal to zero. 
If $\DH_S(\s)=0$, identity~\eqref{eq:freetriangleUS} gives that $f_S(\s)=0$, and thus $S$ has no free triangles. In view of Remark~1 and leveraging the fact that each triangle in $S$ shares edges with two neighboring triangles, it follows by finite induction that all triangles in $S$ all are of the same color and, hence, either $\s_{|S} = \aa_{|S}$ or $\s_{|S} = \bb_{|S}$ or $\s_{|S} = \cc_{|S}$.
To prove the converse direction, note that if $\DH_{S}(\s)>0$, then also $f_S(\s) > 0$, \ie there is at least one free triangle on $S$. Consider the sites corresponding to the three vertices of any such free triangle. By construction they all must be unoccupied and, belonging each to a different partition of $\L$, it follows that $\s_{|S} \neq \aa_{|S},\bb_{|S},\cc_{|S}$.\\
\begin{figure}[!ht]
	\centering
	\subfloat{\includegraphics[scale=1]{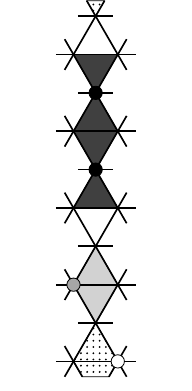}}
	\caption{Hard-core configuration and corresponding block triangles in a vertical stripe of the form $C_{3j}$}
	\label{fig:verticalhorizontalstripesequivalence}
\end{figure}

(ii) Consider a vertical stripe $C$ of the form $C_{3j}$, whose middle column is a subset of $\BB$, see Figure~\ref{fig:verticalhorizontalstripesequivalence}. Analogously to (i), denote by $b_C(\s)$ and $f_C(\s)$ the number of blocked triangles and of free triangles fully contained in the vertical stripe $C$, respectively. There are $2K$ such triangles in total and thus $b_C(\s)+f_C(\s)=2K$. Any particle, regardless of which column/partition it belongs to, blocks exactly two of these triangles, as illustrated in Figure~\ref{fig:verticalhorizontalstripesequivalence}, so that $b_C(\s)=2 \sum_{v \in C} \s(v)$ and, hence,
\begin{equation}
\label{eq:freetriangleUC}
	f_C(\s) = 2 \vdtr - b_C(\s)= 2 \Big ( \vdtr - \sum_{v \in C} \s(v) \Big ) = 2 \cdot \DH_C(\s).
\end{equation}
The latter identity readily implies that $\DH_C(\s) \geq 0$, since $f_C(\s)$ is by construction a non-negative integer.

If $\DH_{C}(\s)>0$, then it follows from~\eqref{eq:freetriangleUC} that there is at least one free triangle fully contained in stripe $C$. The three sites of any such free triangle, each belonging to a different partition of $\L$, must be all unoccupied and thus there cannot be a vertical bridge on $C$. For the reverse implication we argue as follows. By assumption there is at least one black particle and, therefore, two blocked black triangles on $C$. If $\DH_{C}(\s)=0$, then there are$f_C(\s)=0$ free triangles on $C$ in view of~\eqref{eq:freetriangleUC}. Remark 1 states that blocked triangles sharing an edge must be of the same color (cf.~the rightmost case of Figure~\ref{fig:sharingedges}), it readily follows by finite induction that all the triangles on $C$ are black, which means that all the sites in column $c_{3j+1}$ must be occupied, yielding $\s_{|C} \equiv \bb_{|C}$. \qedhere
\end{proof}
\FloatBarrier
We are now ready to state and prove the lower bound on the communication height between any pair of stable configurations.

\begin{prop}[Lower bound on the communication height between $\aa$ and $\bb$]\label{prop:lowerphiaabb}
The communication height between $\aa$ and $\bb$ in the energy landscape corresponding to the hard-core model on the $2K \times 3L$ triangular grid satisfies the following inequality
\[
	\Phi(\aa,\bb) - H(\aa) \geq \min\{\vdtr, 2\hdtr\}+1.
\]
\end{prop}
\begin{proof}
We will show that in every path $\o: \aa \to \bb$ there exists at least one configuration with energy difference greater than or equal to $\min\{\vdtr, 2\hdtr\} +1$.
Consider a path $\o=(\o_1,\dots,\o_n)$ from $\aa$ to $\bb$. Without loss of generality, we may assume that there are no void moves in $\o$, \ie at every step either a particle is added or a particle is removed, so that $H(\o_{i+1}) = H(\o_{i}) \pm 1$ for every $1 \leq i \leq n-1$. Since configuration $\aa$ has no $\bb$--bridges, while $\bb$ does, at some point along the path $\o$ there must be a configuration which is the first to display a $\bb$--bridge, that is a column or a row occupied only by black particles. Let $m^*$ be the index corresponding to such configuration, \ie
\[
	m^*:=\{ m \leq n ~|~ \exists \, i ~:~ (\o_m)_{|r_i} = \bb_{|r_i} \quad \mathrm{ or } \quad \exists \, j ~:~ (\o_m)_{|c_j} = \bb_{|c_j} \}.
\]
Since a $\bb$--bridge cannot be created in only two steps starting from $\aa$, we must have $m^*>2$. We claim that
\[
	\max \{ \DH(\o_{m^*-1}),\DH(\o_{m^*-2}) \}\geq \min\{\vdtr, 2\hdtr\}+1.
\]
Since the addition of a single black particle cannot create more than one bridge in each direction, it is enough to consider the following three cases:
\begin{itemize}
	\item[\textup{(a)}] $\o_{m^*}$ displays a vertical $\bb$--bridge only;
	\item[\textup{(b)}] $\o_{m^*}$ displays a horizontal $\bb$--bridge only;
	\item[\textup{(c)}] $\o_{m^*}$ displays a $\bb$--cross.
\end{itemize}

For case (a), note that configuration $\o_{m^*}$ does not have any horizontal bridge. Indeed, it cannot have a horizontal $\bb$--bridge, otherwise we would be in case (c), and any horizontal $\aa$-- or $\cc$-- bridge cannot coexist with the vertical $\bb$--bridge, in view of Lemma~\ref{lem:crossesaremonochromatic}. Hence, the energy difference of every horizontal stripe is strictly positive, thanks to Lemma~\ref{lem:efficientstripes}(i), and thus
\[
	\DH(\o_{m^*}) = \sum_{i=0}^{\vdtr-1} \DH_{S_i}(\o_{m^*}) \geq \vdtr.
\]
Furthermore, configurations $\o_{m^*-1}$ and $\o_{m^*}$ differ in a unique site $v^* \in \BB$, which is such that $\o_{m^*-1}(v^*)=0$ and $\o_{m^*}(v^*)=1$. Hence, $\DH(\o_{m^*-1})=\DH(\o_{m^*})+1$ and thus
\[
	\DH(\o_{m^*-1}) \geq \vdtr +1.
\]

The argument for case (b) is similar to that of case (a). Firstly, configuration $\o_{m^*}$ does not display any vertical bridge. Lemma~\ref{lem:crossesaremonochromatic} implies that there cannot be any vertical $\aa$-- or  $\cc$--bridge due to the presence of a horizontal $\bb$--bridge, while a vertical $\bb$--bridge cannot exist, otherwise there would be a $\bb$--cross and we would be in case (c). Every vertical stripe has at least one black particle, due to the presence of a horizontal $\bb$--bridge. Hence, $\DH_{C_j}(\o_{m^*})\geq 1$ for every $j=0,\dots,2 \hdtr-1$ in view of Lemma~\ref{lem:efficientstripes}(ii). Therefore,
\[
	\DH(\o_{m^*}) = \sum_{j=0}^{2 \hdtr-1} \DH_{C_j}(\o_{m^*}) \geq 2\hdtr.
\]
From this inequality it follows that $\DH(\o_{m^*-1}) \geq 2 \hdtr +1$, because, as for case (a), the definition of $m^*$ implies $\DH(\o_{m^*-1})=\DH(\o_{m^*})+1$.

Consider now case (c), in which $\o_{m^*}$ displays a $\bb$--cross. The presence of both a vertical and a horizontal $\bb$--bridge means that $\o_{m^*}$ has a black particle in every vertical and horizontal stripe. This property is inherited by the configuration $\o_{m^*-1}$, since it differs from $\o_{m^*}$ only by the removal of the black particle lying at the intersection of the vertical and horizontal bridge constituting the cross. Furthermore, by definition of $m^*$, configuration $\o_{m^*-1}$ cannot have any $\bb$--bridge, neither vertical nor horizontal. These two facts, in combination with Lemma~\ref{lem:efficientstripes}, imply that
\[
	\DH(\o_{m^*-1}) \geq \min\{ \vdtr, 2 \hdtr\}.
\]
If $\DH(\o_{m^*-1}) \geq \min\{ \vdtr, 2 \hdtr\} +1$, then the proof is completed. Otherwise, the energy difference of configuration $\o_{m*-1}$ is $\DH(\o_{m^*-1}) = \min\{ \vdtr, 2 \hdtr\}$. The configuration preceding $\o_{m^*-1}$ in the path $\o$ satisfies
\begin{equation}
\label{eq:omstarminus2pm}
	\DH(\o_{m^*-2}) = \min\{ \vdtr, 2 \hdtr\} \pm 1,
\end{equation}
since it differs from $\o_{m^*-1}$ by a single site update. Suppose first that 
\begin{equation}
\label{eq:omstarminus2tr}
	\DH(\o_{m^*-2}) = \min\{ \vdtr, 2 \hdtr\} - 1.
\end{equation}
This means that $\o_{m^*-2}$ differs from $\o_{m^*-1}$ by the addition of a particle. Therefore, also configuration $\o_{m^*-2}$ has at least one black particle in every horizontal stripe, \ie
\begin{equation}
\label{eq:oneblackparticlehorizontal}
	\sum_{v \in S_i \cap \BB} \o_{m^*-2}(v) \geq 1 \quad \forall \, i =0,\dots, \vdtr,
\end{equation}
and at least one black particle in every vertical stripe, \ie
\begin{equation}
\label{eq:oneblackparticlevertical}
	\sum_{v \in C_j \cap \BB} \o_{m^*-2}(v) \geq 1 \quad \forall \, j =0,\dots, 2\hdtr-1. 
\end{equation}
If $\vdtr \leq 2 \hdtr$, \eqref{eq:omstarminus2tr} and the pigeonhole principle imply that there must be a horizontal stripe $S$ such that $\DH_S(\o_{m^*-2})=0$. In view of~\eqref{eq:oneblackparticlehorizontal} and Lemma~\ref{lem:efficientstripes}(i), $\o_{m^*-2}$ must have a horizontal $\bb$--bridge in $S$, which contradicts the definition of $m^*$.
When instead $\vdtr > 2 \hdtr$, it follows from~\eqref{eq:omstarminus2tr} that there must be a vertical stripe $C$ such that $\DH_C(\o_{m^*-2})=0$. Also in this case, \eqref{eq:oneblackparticlevertical} and Lemma~\ref{lem:efficientstripes}(ii) imply that $\o_{m^*-2}$ displays a vertical $\bb$--bridge in $C$, in contradiction with the definition of $m^*$. We have in this way proved that assumption~\eqref{eq:omstarminus2tr} always leads to a contradiction, so in view of~\eqref{eq:omstarminus2pm} we have $\DH(\o_{m^*-2}) = \min\{ \vdtr, 2 \hdtr \} + 1$ and the proof is concluded also for case (c). 
\end{proof}

\subsection{Reference path and absence of deep cycles}
\label{sub52}
In this subsection we describe an iterative procedure that constructs a path from a suitable initial configuration to a target stable configuration. We will refer to it as \textit{energy reduction algorithm} since the yielded path $\o$ brings the initial configuration $\s$ to a configuration with lower (in particular, minimum) energy while guaranteeing that the energy along the path will never exceed the initial value plus one or two, depending on the structure of the initial configuration.
These two algorithmic procedures modify the initial configuration using only moves allowed by the hard-core dynamics (i.e., single-site updates) and increasingly grow a uniform cluster (aligned with the target configuration) proceeding either row by row or column by column.
These two variations, despite being similar in spirit, will be described separately, since the structure of horizontal and vertical stripes of the triangular grid is fundamentally different. Nonetheless, 
the core mechanisms of both these algorithms is the same: orderly add particles aligned with the target configuration and, if necessary, remove the particles on the other two partitions that block the growth of such a cluster.
In order for $\s$ to be a suitable starting configuration for the energy reduction algorithm, $\s$ should have ``enough room'' for such a cluster to be created, condition that is guaranteed when all the occupied sites in two adjacent rows (or columns) belong to the same partition.

Such energy reduction algorithms will be used in Proposition~\ref{prop:refpathaatobb} to construct the reference path from $\aa$ to $\bb$ and to show the absence of deep cycle in the state space $\cX$, the crucial step in the proof of Theorem~\ref{thm:structuralproperties}(iii).


\subsubsection*{Energy reduction algorithm by rows}
We now describe in detail the energy reduction algorithm by rows with $\bb$ as target configuration. In order for $\s \in \cX$ to be a suitable initial configuration for this iterative procedure, we require that $\s$ has no gray or white particles in the first horizontal stripe $S_0 = r_0 \cup r_1$, \ie
\begin{equation}
\label{eq:initialconditionRAR}
	\s(v) = 0 \quad \forall \, v \in S_0 \cap (\AA \cup \CC).
\end{equation}
Figure~\ref{fig:firststripefree} shows a hard-core configuration that satisfies this initial condition.
\begin{figure}[!ht]
	\centering
	\includegraphics[scale=0.95]{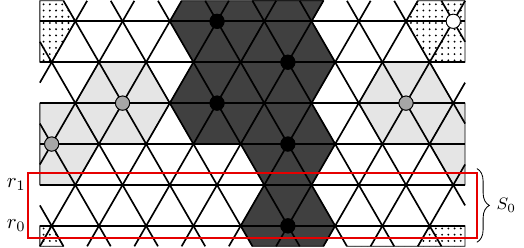}
	\caption{Example of a hard-core configuration on the $6 \times 9$ triangular grid that satisfies~\eqref{eq:initialconditionRAR}}
	\label{fig:firststripefree}
\end{figure}
\FloatBarrier
The output of this algorithm is a path $\o$ from $\s$ to $\bb$, which we construct as the concatenation of $2\vdtr$ paths $\o^{(1)},\dots,\o^{(2\vdtr)}$. For every $i=1,\dots,2\vdtr$, path $\o^{(i)}$ goes from $\s_i$ to $\s_{i+1}$, where we set $\s_{1}:=\s$, $\s_{2K+1}:=\bb$ and define for $i=2,\dots,2\vdtr$
\[
	\s_{i}(v):=
	\begin{cases}
		\bb(v) 	& \text{ if } v \in r_1,\dots,r_{i-1},\\
		0 			& \text{ if } v \in r_{i} \cap (\AA \cup \CC),\\
		\s(v) 	& \text{ if } v \in r_{i} \cap \BB \text{ or } v \in r_{i+1},\dots,r_{2K-1}.\\
	\end{cases}
\]
We now describe in detail how to construct each of the paths $\o^{(i)}$ for $i=1,\dots,2\vdtr$. Each path $\o^{(i)}=(\o^{(i)}_{1}, \dots, \o^{(i)}_{2\hdtr+1})$ comprises $2\hdtr+1$ moves (but possibly void) and is such that $\o^{(i)}_{1}=\s_i$ and $\o^{(i)}_{2\hdtr+1}=\s_{i+1}$. We start from configuration $\smash{\o^{(i)}_0=\s_i}$ and we repeat iteratively the following procedure for all $j=1,\dots,2\hdtr$:
\begin{itemize}
	\item If $j \equiv 1 \pmod 2$, consider the pair of sites $v \in \AA$ and $v'\in \CC$ defined by
	\[
	\begin{cases}
		v=(i+1, 3j), \quad \quad v'=(i+1,3j+ 2)		& \text{ if } i \equiv 0 \pmod 2,\\
		v=(i+1, 3j-3), \, \, v'=(i+1,3j -1)			 		& \text{ if } i \equiv 1 \pmod 2.
	\end{cases}
	\]
	Note that the two sites $v$ and $v'$ are always neighbors, so that only one of the two can be occupied.
	\begin{itemize}
		\item[-] If $\o^{(i)}_{j}(v)=0=\o^{(i)}_{j}(v')$, we set $\o^{(i)}_{j+1}=\o^{(i)}_{j}$, so $H(\o^{(i)}_{j+1}) = H(\o^{(i)}_{j})$.
		\item[-] If $\smash{\o^{(i)}_{j}(v)=1}$ or $\smash{\o^{(i)}_{j}(v')=1}$, then we remove from configuration $\o^{(i)}_{j}$ the particle in the unique occupied site between $v$ and $v'$, increasing the energy by $1$ and obtaining in this way configuration $\smash{\o^{(i)}_{j+1}}$, which is such that $\smash{H(\o^{(i)}_{j+1}) = H(\o^{(i)}_{j})+1}$.
	\end{itemize}
	\item If $j \equiv 0 \pmod 2$, consider the site $v \in \BB$ defined as
	\[
		v=
		\begin{cases}
			(i,3j -2)		& \text{ if } i \equiv 0 \pmod 2,\\
			(i,3j -5)		& \text{ if } i \equiv 1 \pmod 2.
		\end{cases}
	\]	
	\begin{itemize}
		\item[-] If $\o^{(i)}_{j}(v)=1$, we set $\o^{(i)}_{j+1}=\o^{(i)}_{j}$ and thus $H(\o^{(i)}_{j+1}) = H(\o^{(i)}_{j})$.
		\item[-] If $\o^{(i)}_{j}(v)=0$, then we add to configuration $\o^{(i)}_{j}$ a particle in site $v$ decreasing the energy by $1$. We obtain in this way a configuration $\smash{\o^{(i)}_{j+1}}$, which is a hard-core configuration because by construction all the first neighboring sites of $v$ are unoccupied. In particular, the two particles residing in the two sites above $v$ may have been removed exactly at the previous step. The new configuration has energy $\smash{H(\o^{(i)}_{j+1}) = H(\o^{(i)}_{j})-1}$.
	\end{itemize}
\end{itemize}
The way the path $\o^{(i)}$ is constructed shows that
$
	H(\s_{i+1}) \leq H(\s_i)
$
for every $i=1,\dots,2\vdtr$, since the number of particles added in row $r_i$ is greater than or equal to the number of particles removed in row $r_{i+1}$. Moreover,
$
	\Phi_{\o^{(i)}} \leq H(\s_i) + 1,
$
since along the path $\o^{(i)}$ every particle removal (if any) is always followed by a particle addition. These two properties imply that the path $\o: \s \to \bb$ created by concatenating $\o^{(1)},\dots,\o^{(2\vdtr)}$ satisfies
\[
	\Phi_{\o} \leq H(\s) + 1.
\]

Note that the energy reduction algorithm by rows can be tweaked in order to have either $\aa$ or $\cc$ as target configuration. In particular, the condition~\eqref{eq:initialconditionRAR} for the initial configuration $\s$ should be adjusted accordingly, requiring that $\s$ has no black and white (black and gray, respectively) particles in the first horizontal stripe $S_0$, depending on whether the target configuration is $\aa$ or $\cc$, respectively.

\subsubsection*{Energy reduction algorithm by columns}
We now illustrate how the energy reduction algorithm by columns works choosing $\bb$ as target configuration. Note that the procedure we are about to describe can be tweaked to yield a path with target configuration $\aa$ or $\cc$, but we omit the details. If the target configuration is $\bb$, we require that the initial configuration $\s \in \cX$ has no particles on columns $c_2$ and $c_3$, namely
\begin{equation}
\label{eq:initialconditionRAC}
	\s(v) = 0 \quad \forall \, v \in c_{2} \cup c_{3}.
\end{equation}
Since $c_{2} = C_0 \cap \CC$ and $c_{3} = C_1 \cap \AA$, condition~\eqref{eq:initialconditionRAC} requires there are no white particles in $C_0$ and no gray particles in $C_{1}$. Figure~\ref{fig:c2c3free} shows a hard-core configuration that satisfies this initial condition.
\begin{figure}[!ht]
	\centering
	\includegraphics[scale=1]{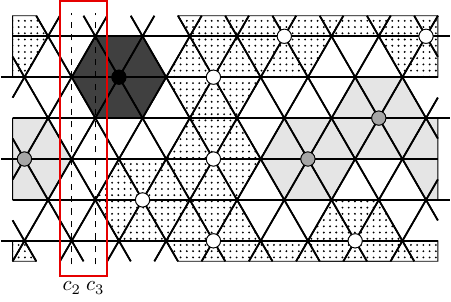}
	\vspace{-0.1cm}
	\caption{Example of hard-core configuration on the $6 \times 9$ triangular grid that satisfies~\eqref{eq:initialconditionRAC}}
	\label{fig:c2c3free}
	\vspace{-0.1cm}
\end{figure}
\FloatBarrier
The output of this algorithm is a path $\o$ from $\s$ to $\bb$, which we construct as concatenation of $2\hdtr$ paths $\o^{(1)},\dots,\o^{(2 \hdtr)}$. For every $j=1,\dots,2\hdtr$, path $\o^{(j)}$ goes from $\s_j$ to $\s_{j+1}$, where we set $\s_{1}:=\s$, $\s_{2L+1}:=\bb$ and define for $j=2,\dots,2\hdtr$
\[
	\s_{j}(v):=
	\begin{cases}
		\bb(v) 		& \text{ if } v \in c_2,\dots,c_{3j},\\
		\s(v) 	& \text{ if } v\in c_{3j+1}, \dots, c_{6L+1}.\\
	\end{cases}
\]
We now describe in detail how to construct each of the paths $\o^{(j)}$ for $j=1,\dots,2\hdtr$.
We distinguish two cases, depending on whether (a) $\s_j$ has a vertical bridge in column $c_{3j+2}$ or (b) not, see the two examples in Figure~\ref{fig:RACab}.  

\begin{figure}[!ht]
	\centering
	\vspace{-0.1cm}
	\subfloat[The configuration $\s_1$ has a $\cc$--bridge on column $c_8$\label{fig:RACa}]{\includegraphics[scale=1]{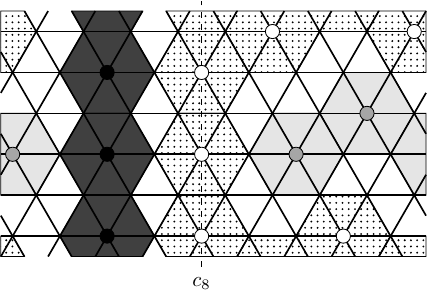}}
	\hspace{0.9cm}
	\subfloat[The configuration $\s_2$ has no $\cc$--bridges on column $c_{11}$\label{fig:RACb}]{\includegraphics[scale=1]{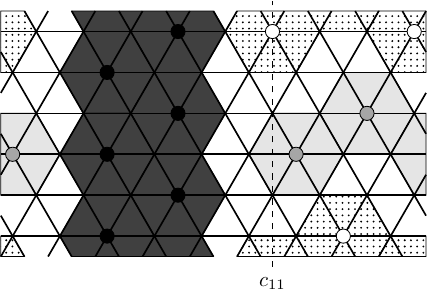}}
	\vspace{-0.1cm}
	\caption{The configurations $\s_1$ and $\s_2$ corresponding to the initial configuration $\s$ in Figure~\ref{fig:c2c3free}}
	\label{fig:RACab}
\end{figure}
\FloatBarrier

Consider case (a) first. First notice that the presence of a vertical ($\cc$--)bridge in column $c_{3j+2}$ implies that all sites of the adjacent column $c_{3j+3}$ must be empty in configuration $\s_j$.

We construct a path $\o^{(j)}=(\o^{(j)}_{1}, \dots, \o^{(j)}_{2\vdtr+1})$ of length $2\vdtr+1$ (but possibly comprising void moves), with $\o^{(j)}_{1}=\s_j$ and $\o^{(j)}_{2\vdtr+1}=\s_{j+1}$. Denote by $o(j) \in \{0,1\}$ the integer number such that $o(j) \equiv j \pmod 2$. We first remove the two white particles in column $c_{3j+2}$ that lie in row $r_{o(j)}$ and $r_{o(j)+2}$ in two successive steps, obtaining in this way configuration $o^{(j)}_3$, which is such that $H(\o^{(j)}_{3}) = H(\o^{(j)}_{1})+2$. We then repeat iteratively the following procedure to obtain the configuration $\o^{(j)}_{i+1}$ from $\o^{(j)}_{i}$ for all $i=3,\dots,2\vdtr-1$:
\begin{itemize}
	\item If $i \equiv 1 \pmod 2$, consider the site $v \in c_{3j+1} \subset \BB$ with coordinates $(3j+1,o(j)+i-2)$ and add a (black) particle there, obtaining in this way configuration $\o^{(j)}_{i+1}$. Such a particle can be added since all its six neighboring sites are empty. More specifically, the three left ones have been (possibly) emptied along the path $\o^{(j-1)}$, while the one in $c_{3j+3}$ is empty by assumption and the other two sites on its right have been emptied in the previous steps of $\o^{(j)}$. Since we added one particle, $H(\o^{(j)}_{i+1}) = H(\o^{(j)}_{i})-1$.
	\item If $i \equiv 0 \pmod 2$, consider the site $v \in c_{3j+2} \subset \CC$ with coordinates $(3j+2,o(j)+i)$ and remove the (white) particle lying there, obtaining in this way configuration $\o^{(j)}_{i+1}$, which is such that $H(\o^{(j)}_{i+1}) = H(\o^{(j)}_{i})+1$.
\end{itemize}
This procedure outputs configuration $\o^{(j)}_{2\vdtr}$ which has no white particles in column $c_{3j+2}$ and an empty site in column $c_{3j+1}$, the one with coordinates $(3j+1,2\vdtr-1-o(j))$. All the neighboring sites of this site are empty by construction and, adding a black particle in this site, we obtain configuration $\o^{(j)}_{2\vdtr+1}=\s_{j+1}$, which is such that $H(\s_{j+1}) = H(\o^{(j)}_{2\vdtr})-1$. The way the path $\o^{(j)}$ is constructed shows that
\[
	H(\s_{j+1}) = H(\s_j),
\]
since we added exactly $\vdtr$ (black) particles in column $c_{3j+1}$ and removed exactly $\vdtr$ (white) particles in columns $c_{3j+2}$. Moreover,
\begin{equation}
\label{eq:phicase1}
	\Phi_{\o^{(j)}} = \max_{\h \in \o^{(j)}} H(\h) = H(\s_j) + 2
\end{equation}
since along the path $\o^{(j)}$ every particle removal is followed by a particle addition, except at the beginning when we remove two particles consecutively.

Consider now case (b). We claim that, since there is no vertical ($\cc$--)bridge in column $c_{3j+2}$, there exists a site $v^*$ in column $c_{3j+1}$ with at most one neighboring occupied site. First of all, all sites in column $c_{3j}$ and $c_{3j-1}$ have been emptied along the path $\o^{(j-1)}$, so all sites in $c_{3j+1}$ have no left neighboring sites occupied. Let us look now at the right neighboring sites. Since there is no vertical $\cc$--bridge in column $c_{3j+2}$, there exists an empty site in it, say $w$. Modulo relabeling the rows, we may assume that $w$ has coordinates $(3j+2,o(j))$, where $o(j)$ is the integer in $\{0,1\}$ such that $o(j) \equiv j \pmod 2$. The site $v^*=(3j+1,o(j)+1)$ has then the desired property, since at most one of its two remaining right neighboring sites (those with coordinates $(3j+2,o(j)+2)$ and $(3j+3,o(j)+1)$, respectively) can be occupied, since they are also neighbors of each other.

We construct a path $\o^{(j)}=(\o^{(j)}_{1}, \dots, \o^{(j)}_{2\vdtr+1})$ of length $2\vdtr+1$ (but possibly comprising void moves), with $\o^{(j)}_{1}=\s_j$ and $\o^{(j)}_{2\vdtr+1}=\s_{j+1}$. We then repeat iteratively the following procedure to obtain configuration $\o^{(j)}_{i+1}$ from $\o^{(j)}_{i}$ for all $i=1,\dots,2\vdtr$:
\begin{itemize}
	\item If $i \equiv 1 \pmod 2$, consider the two sites $(3j+2,o(j)+i+1) \in \CC$ and $(3j+3,o(j)+i) \in \AA$. Since they are neighboring sites, at most one of them is occupied. If they are both empty, we set $\o^{(j)}_{i+1}= \o^{(j)}_{i}$. If instead there is a particle in either of the two, we remove it, obtaining in this way configuration $\o^{(j)}_{i+1}$, which is such that $H(\o^{(j)}_{i+1}) = H(\o^{(j)}_{i})+1$.
	\item If $i \equiv 0 \pmod 2$, consider the site $v \in c_{3j+1} \subset \BB$ with coordinates $(3j+1,o(j)+i-1)$ and add a (black) particle there, obtaining in this way configuration $\o^{(j)}_{i+1}$. Such a particle can be added since all its six neighboring sites are empty. More specifically, the three left ones have been (possibly) emptied along the path $\o^{(j-1)}$, while the other two sites on its right have been emptied in the previous step of $\o^{(j)}$. Since we added one particle, $H(\o^{(j)}_{i+1}) = H(\o^{(j)}_{i})-1$.
\end{itemize}
The way the path $\o^{(j)}$ is constructed shows that
$
	H(\s_{j+1}) \leq H(\s_j),
$
since the number of (black) particles added in column $c_{3j+1}$ is greater than or equal to the number of (white/gray) particles removed in columns $c_{3j+2}$ and $c_{3j+3}$. Moreover, along the path $\o^{(j)}$ every particle removal (if any) is always followed by a particle addition, and hence
\begin{equation}
\label{eq:phicase2}
	\Phi_{\o^{(j)}} = \max_{\h \in \o^{(j)}} H(\h) \leq H(\s_j) + 1.
\end{equation}
Consider now the path $\o: \s \to \bb$ created by concatenating $\o^{(1)},\dots,\o^{(2\hdtr)}$, which are constructed either using the procedure in case (a) or that in case (b). 
First notice that, regardless of which procedure has been used at step $j$, the inequality $H(\s_{j+1}) \leq H(\s_j)$ holds for every $j=1,\dots,2\hdtr$. Using this fact in combination with~\eqref{eq:phicase1} and~\eqref{eq:phicase2} shows that the path $\o$ always satisfies
\[
	\Phi_{\o} \leq H(\s) + 2.
\]
Furthermore, in the special case in which $\s$ has no vertical $\cc$--bridges, our procedure considers case (b) for every $j=1,\dots,2 \hdtr$ and thus, by virtue of~\eqref{eq:phicase2}, the path $\o$ satisfies
\[
	\Phi_{\o} - H(\s) \leq 1.
\]

If the target configuration of the energy reduction algorithm by columns is the configuration $\aa$ (or $\cc$) one should adjust the condition~\eqref{eq:initialconditionRAC} on the initial condition accordingly, requiring that $\s$ has no particles in columns $c_1$ and $c_2$ (columns $c_0$ and $c_1$, respectively). The offset of rows and columns in the procedures described above should of course be tweaked appropriately.

We now use the energy reduction algorithms we just introduced to show that the lower bound for the communication height between $\aa$ and $\bb$ given in Proposition~\ref{prop:lowerphiaabb} is sharp, by explicitly giving a path that attains that value.
\begin{prop}[Reference path]\label{prop:refpathaatobb}
In the energy landscape corresponding to the hard-core model on the $2K \times 3L$ triangular grid there exists a path $\o^*: \aa \to \bb$ in $\cX$ such that
\[
	\Phi_{\o^*} - H(\aa) = \min\{\vdtr, 2\hdtr\}+1.
\]
\end{prop}
\begin{proof}
We distinguish two cases, depending on whether (a) $\vdtr \leq 2 \hdtr$ and (b) $\vdtr > 2 \hdtr$. In either case we first construct a path $\o^{(1)}: \aa \to \s^*$ where $\s^*$ is a configuration to which we can apply energy reduction algorithm by columns (rows, respectively), and then, using this latter, we produce a path $\o^{(2)}: \s^* \to \bb$. The desired path $\o^*: \aa \to \bb$ will then be the concatenation of the paths $\o^{(1)}$ and $\o^{(2)}$.

Figure~\ref{fig:aatobbfirst} illustrates the reference path from $\aa$ to $\bb$ in case (a) for the $6 \times 9$ triangular grid, while Figure~\ref{fig:aatobbsecond} depicts some snapshots of $\o^*: \aa \to \bb$ in case (b) for the $10 \times 6$ triangular grid.

For case (a), the configuration $\s^*$ differs from $\aa$ only in the sites of column $c_3$ and, specifically,
\[
	\s^*(v):=
	\begin{cases}
		\aa(v) 	& \text{ if } v \in V \setminus c_3,\\
		0 		& \text{ if } v \in c_3.
	\end{cases}
\]
The path $\o^{(1)}=(\o^{(1)}_{1},\dots,\o^{(1)}_{\vdtr+1})$, with $\o^{(1)}_{1}=\aa$ and $\o^{(1)}_{\vdtr+1}=\s^*$ can be constructed as follows. For $i=1,\dots,\vdtr$, at step $i$ we remove from configuration $\o^{(1)}_{i}$ the particle in the site of coordinates $(3,2i-1)$, increasing the energy by $1$ and obtaining in this way configuration $\o^{(1)}_{i+1}$. Therefore the configuration $\s^*$ is such that 
$
	H(\s^*)-H(\aa) = \vdtr
$
and 
$
	\Phi_{\o^{(1)}} = H(\s^*) = H(\aa)+ \vdtr.
$

The second path $\o^{(2)}: \s^* \to \bb$ is then constructed by means of the energy reduction algorithm by columns, which can be used since the configuration $\s^*$ satisfies condition~\eqref{eq:initialconditionRAC} and hence is a suitable initial configuration for the algorithm. Since configuration $\s^*$ has no vertical $\cc$--bridges (see case (b) for the energy reduction algorithm by columns), the procedure guarantees that
\[
	\Phi_{\o^{(2)}} = H(\s^*) +1 = H(\aa) + \vdtr +1,
\]
and, therefore, $\Phi_{\o^*} = \max\{\Phi_{\o^{(1)}},\Phi_{\o^{(2)}}\} = H(\aa)+ \vdtr+1$ as desired. 

\captionsetup[subfigure]{labelformat=empty}
\begin{figure}[!ht]
\centering
	\subfloat[$\aa$]{\includegraphics{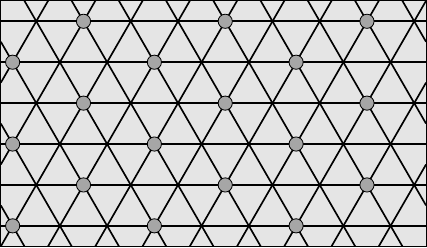}}
	\hspace{0.5cm}
	\subfloat[$\s^*=\o^{(2)}_1$]{\includegraphics{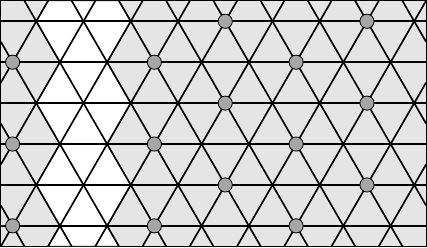}}
	\\
	\subfloat[$\o^{(2)}_2$]{\includegraphics{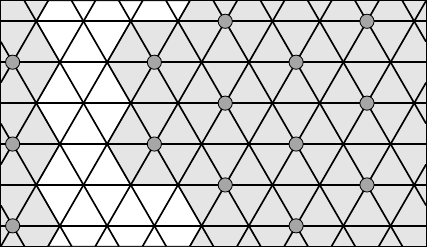}}
	\hspace{0.5cm}
	\subfloat[$\o^{(2)}_3$]{\includegraphics{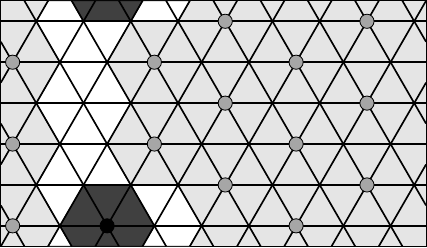}}
	\\
	\subfloat[$\o^{(2)}_4$]{\includegraphics{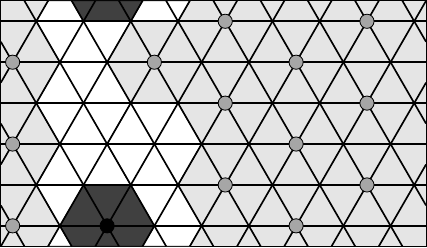}}
	\hspace{0.5cm}
	\subfloat[$\o^{(2)}_5$]{\includegraphics{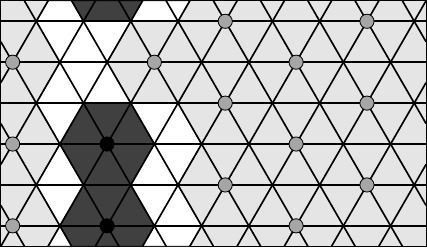}}
	\\
	\subfloat[$\o^{(2)}_6$]{\includegraphics{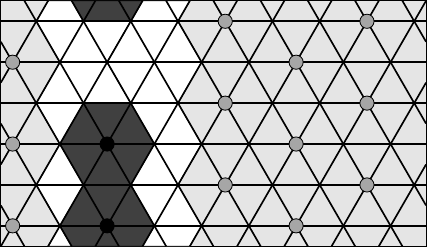}}
	\hspace{0.5cm}
	\subfloat[$\o^{(2)}_7$]{\includegraphics{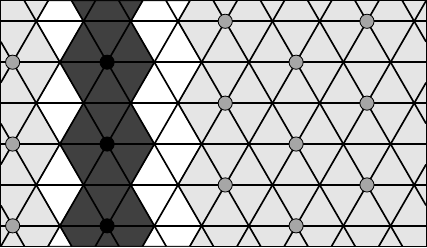}}
	\\
	\subfloat[$\o^{(2)}_{13}$]{\includegraphics{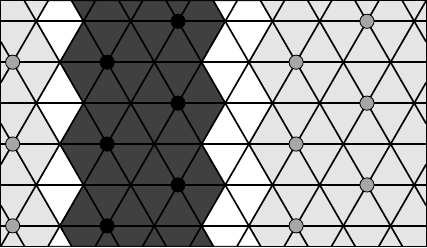}}
	\hspace{0.5cm}
	\subfloat[$\bb$]{\includegraphics{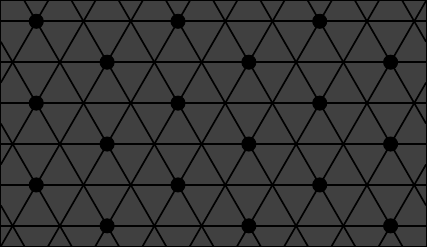}}
	\caption{Illustration of the reference path $\o^*: \aa \to \bb$ in the case $\vdtr \leq 2\hdtr$}
	\label{fig:aatobbfirst}
\end{figure}

\begin{figure}[!ht]
\centering
	\subfloat[$\aa$]{\includegraphics{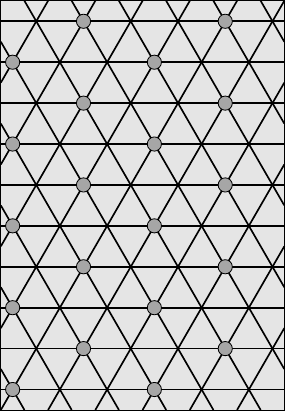}}
	\hspace{0.3cm}
	\subfloat[$\s^*=\o^{(2)}_1$]{\includegraphics{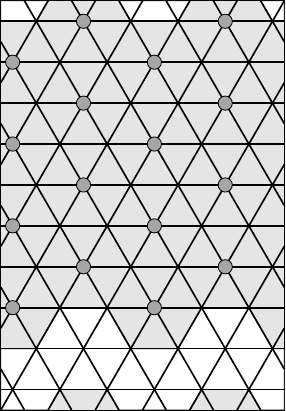}}
	\hspace{0.3cm}
	\subfloat[$\o^{(2)}_2$]{\includegraphics{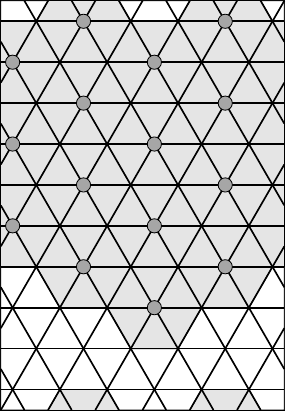}}
	\\
	\subfloat[$\o^{(2)}_3$]{\includegraphics{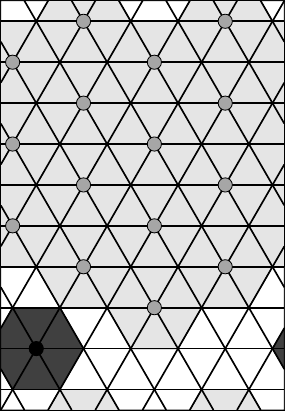}}
	\hspace{0.3cm}
	\subfloat[$\o^{(2)}_4$]{\includegraphics{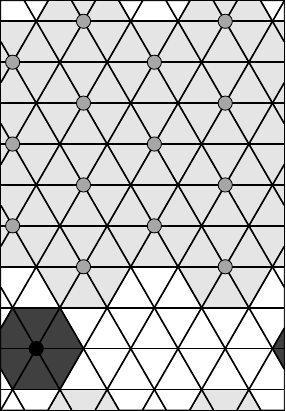}}
	\hspace{0.3cm}
	\subfloat[$\o^{(2)}_5$]{\includegraphics{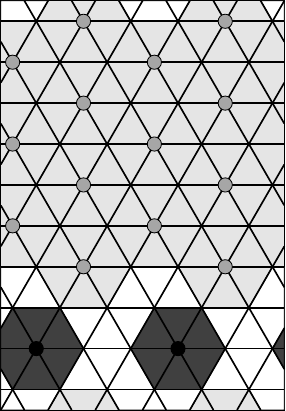}}
	\\
	\subfloat[$\o^{(2)}_9$]{\includegraphics{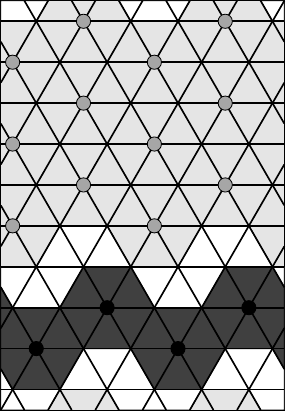}}
	\hspace{0.3cm}
	\subfloat[$\o^{(2)}_{13}$]{\includegraphics{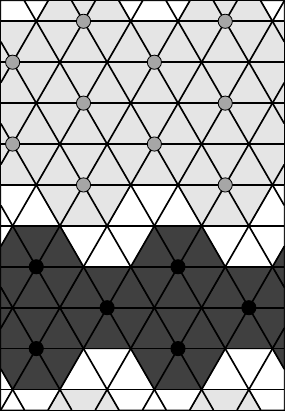}}
	\hspace{0.3cm}
	\subfloat[$\bb$]{\includegraphics{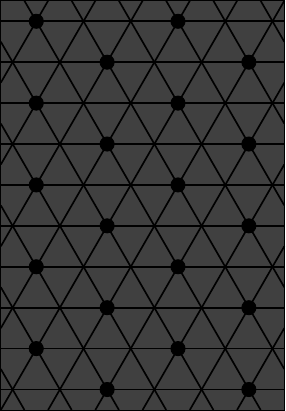}}
	\caption{Illustration of the reference path $\o^*: \aa \to \bb$ in the case $\vdtr > 2 \hdtr$}
	\label{fig:aatobbsecond}
\end{figure}

For case (b), consider the configuration $\s^*$ that differs from $\aa$ only in the sites of the first horizontal stripe $S_0$, namely
\[
	\s^*(v):=
	\begin{cases}
		\aa(v) 	& \text{ if } v \in V \setminus S_0,\\
		0 		& \text{ if } v \in S_0.
	\end{cases}
\]
The path $\o^{(1)}=(\o^{(1)}_{1},\dots,\o^{(1)}_{2\hdtr+1})$, with $\o^{(1)}_{1}=\aa$ and $\o^{(1)}_{2\hdtr+1}=\s^*$ can be constructed as follows. For $i=1,\dots,2\hdtr$, at step $i$ we remove from configuration $\o^{(1)}_{i}$ the first particle in lexicographic order in $S_0$, increasing the energy by $1$ and obtaining in this way configuration $\o^{(1)}_{i+1}$. Therefore the configuration $\s^*$ is such that 
$
	H(\s^*)-H(\aa) = 2\hdtr
$
and 
$
	\Phi_{\o^{(1)}} = H(\s^*) = H(\aa)+ 2\hdtr.
$

The second path $\o^{(2)}: \s^* \to \bb$ is then constructed by means of the energy reduction algorithm by rows, which can be used since the configuration $\s^*$ satisfies condition~\eqref{eq:initialconditionRAR} and hence is a suitable initial configuration for the algorithm. The energy reduction algorithm by rows guarantees that
\[
	\Phi_{\o^{(2)}} = H(\s^*) +1 = H(\aa) + 2\hdtr +1.
\]
and thus the conclusion follows, since $\Phi_{\o^*} = \max\{\Phi_{\o^{(1)}},\Phi_{\o^{(2)}}\} = H(\aa)+ 2\hdtr+1$.
\end{proof}
\FloatBarrier

\subsection{Proof of Theorem~\ref{thm:structuralproperties}}
\label{sub53}

\hspace{\parindent} (i) In every hard-core configuration $\s \in \cX$ each particle blocks exactly six triangles, so that the total number $b_\L(\s)$ of blocked triangles in $\s$ is given by
\begin{equation}
\label{eq:trH}
	b_\L(\s)=6 \sum_{v \in V} \s(v) = - 6 H(\s).
\end{equation}
Since $\L$ has $12 K L$ triangles in total and we must have $b_\L(\s) \leq 12 K L$, it readily follows that
$
	\max_{\s \in \cX} \sum_{v \in V} \s(v) \leq 2 \vdtr\hdtr.
$
Configurations $\aa,\bb$ and $\cc$ attain this value in view of~\eqref{eq:componentsizestriangulargrid} and, hence, $\min_{\s \in \cX} H(\s) = -2 K L$.

Suppose by contradiction that there exists another configuration $\s \in \cX \setminus \{ \aa, \bb,\cc\}$ such that $ H(\s) = -2 K L$. In view of~\eqref{eq:trH}, this means that all the triangles are blocked in $\s$. Starting from any triangle and using iteratively the fact that blocked triangles sharing an edge must be of the same color (cf.~Remark 1), it is easy to show by induction that all the blocked triangles are of the same color and thus $\s \in \{ \aa, \bb,\cc\}$, which is a contradiction.\\

(ii) The proof of the identity involving $\Phi(\aa,\bb)-H(\aa)$ readily follows by combining the lower bound in Proposition~\ref{prop:lowerphiaabb} and the statement of Proposition~\ref{prop:refpathaatobb}; the remaining identities immediately follows from symmetry of the triangular grid.\\

(iii) We will show that for every hard-core configuration $\s$ on the $2K \times 3L$ triangular grid with $\s \neq \aa,\bb,\cc$, there exists a path $\o$ from $\s$ to one of the three stable configurations such that
\[
	\Phi_\o -H(\s) \leq \min \{ \vdtr, 2\hdtr\}.
\]
The idea is to construct such a path using the geometric features of the configuration $\s$ and exploiting the energy reduction algorithms described earlier in this section. We distinguish two cases: (a) $\vdtr \leq 2 \hdtr$ and (b) $\vdtr > 2 \hdtr$.

Consider case (a) first, where $\vdtr \leq 2 \hdtr$. We distinguish two sub-cases, depending on whether $\s$ has at least one vertical bridge or not.

If $\s$ has a vertical bridge in a vertical stripe $C$, then $\s$ is a suitable starting configuration for the energy reduction algorithm, which yields a path $\o$ that goes from $\s$ to the stable configuration in $\{\aa,\bb,\cc\}$ on which $\s$ agrees in stripe $C$. The path $\o$ constructed in this way is such that $\Phi_{\o} - H(\s) \leq 2$ and thus $\Phi(\s, \{\aa,\bb,\cc\}) - H(\s) \leq 2 \leq \vdtr \leq \min\{\vdtr, 2\hdtr\}$, since by assumption $\vdtr$ is an integer greater than $1$.
	
Suppose now that there are no vertical bridges in $\s$. Since $\s \not\in \{\aa,\bb,\cc\}$, which is the set of stable configurations in view of Theorem~\ref{thm:structuralproperties}(i), configuration $\s$ has a positive energy difference $\DH(\s)>0$. In view of~\eqref{eq:energywastagesumtr}, this means that there exists a vertical stripe $C^*$ such that $\DH_{C^*}(\s)>0$. Without loss of generality, we may assume (modulo relabeling) that $C^*$ is the vertical stripe $C_1$, which consists of columns $c_1, c_2$ and $c_3$. By definition of energy difference in a stripe, it follows that $\s$ has at most $\vdtr-1$ particles. Removing all the gray and white particles one by one, we construct a path $\o^{(1)}$ from $\s$ to a new configuration $\s^*$ defined as
\[
	\s^*(v) :=
	\begin{cases}
		\s(v) 	& \text{ if } v \in V \setminus (c_2 \cup c_3),\\
		0 			& \text{ if } v \in c_2 \cup c_3.	
	\end{cases}
\]
Since $\s$ has at most $\vdtr-1$ particles in the vertical stripe $C_1$, it follows that
\begin{equation}
\label{eq:sstarnobridges}
	H(\s^*)-H(\s)\leq K-1 \quad \text{ and } \quad \Phi_{\o^{(1)}} - H(\s) \leq \vdtr-1.
\end{equation}
Since we remove all the gray particles from $c_3$ and all the white particles from $c_2$, $\s^*$ is a suitable starting configuration for the energy reduction algorithm by columns with target configuration $\bb$, in view of~\eqref{eq:initialconditionRAC}. We obtain in this way a second path $\o^{(2)}: \s^* \to \bb$, which is such that
\begin{equation}
\label{eq:phio2nobridges}
	\Phi_{\o^{(2)}} - H(\s^*) \leq 1,
\end{equation}
thanks to the absence of vertical bridges in $\s$ (and thus in $\s^*$). In view of~\eqref{eq:sstarnobridges} and~\eqref{eq:phio2nobridges}, the path $\o: \s \to \bb$ obtained by concatenating $\o^{(1)}$ and $\o^{(2)}$ is such that 
$
	\Phi_{\o} - H(\s) \leq \vdtr,
$
and hence $\Phi(\s, \{\aa,\bb,\cc\}) - H(\s) \leq \vdtr$. 

We remark that there is nothing special about $\bb$ as target configuration of the path $\o$ we just constructed. Indeed, by choosing the vertical stripe $C^*$ with a different offset, we could have obtained a configuration $\s^*$ which would have been a suitable initial configuration for the energy reduction algorithm by columns with target configuration $\aa$ or $\cc$.

We now turn to case (b), in which $\vdtr > 2 \hdtr$. Thanks to Lemma~\ref{lem:efficientstripes}(i), there must be a horizontal stripe $S$ on which $\s$ does not have a horizontal bridge, otherwise $\s \in \{\aa,\bb,\cc\}$. In particular, $\s$ has at most $2\hdtr-1$ particles on $S$, which without loss of generality we may assume to be $S_0$. We construct a path $\o^{(1)}$ from $\s$ to a new configuration $\s^*$ by removing all these particles one by one, so that $\Phi_{\o^{(1)}} - H(\s) \leq 2\hdtr-1$ and $H(\s^*)-H(\s) \leq 2\hdtr-1$. Starting with configuration $\s^*$ we can then use the energy reduction algorithm by rows to obtain a second path $\o^{(2)}$ from $\s^*$ to any of the three stable configurations. Since $\Phi_{\o^{(2)}} - H(\s^*) \leq 1$, the path $\o$ constructed by the concatenation of $\o^{(1)}$ and $\o^{(2)}$ satisfies 
$
	\Phi_{\o} - H(\s) \leq 2\hdtr
$
and thus $\Phi(\s, \{\aa,\bb,\cc\}) - H(\s) \leq 2 \hdtr$. \qed

\section{Proofs of the main results}
\label{newsec4}
This section is devoted to the proof of the two main results of the paper, namely Theorems~\ref{thm:tabc} and~\ref{thm:mix_triangularlattice}

We first briefly recall in Subsection~\ref{newsub41} some model-independent results derived in~\cite{NZB16} valid for any Metropolis Markov chain. We show how these general results can be used in combination with the structural properties of the energy landscape of the hard-core model on triangular grids, outlined in Theorem~\ref{thm:structuralproperties}, to prove statements (i), (ii), and (iii) of Theorem~\ref{thm:tabc} in Subsection~\ref{newsub42} and Theorem~\ref{thm:mix_triangularlattice} in Subsection~\ref{newsub43}.
Although statements (iii) and (iv) of Theorem~\ref{thm:tabc} both concern the asymptotic exponentiality of the scaled hitting times and look alike, their proofs slightly differ and for this reason that of statement (iv) is presented separately, in Subsection~\ref{newsub44}, leveraging the symmetries that the state space $\cX$ inherits from the non-trivial automorphisms of the graph $\L$.

\subsection{Model-indepedent results for Metropolis Markov chains}
\label{newsub41}

We present here the model-independent results of the general framework developed in~\cite{NZB16} only in a special case that is relevant for the tunneling times $\tab$ and $\tabc$ under analysis. The more general statements can be found in~\cite{NZB16}, see Corollary~3.16, Theorem~3.17 and~3.19, and Proposition~3.18 and~3.20 therein.

\begin{prop}[Hitting time asymptotics~\cite{NZB16}] \label{prop:modindep}
Consider a non-empty subset $A \subset \cX$ and $\s \in \cX \setminus A$ and the following two conditions: 
\begin{equation}
\label{eq:suffcondPE}
	\Phi(\s,A) - H(\s) = \max_{\h \in \cX \setminus A} \Phi(\h,A) - H(\h),
\end{equation}
and
\begin{equation}
\label{eq:suffcondAE}
	\Phi(\s,A) - H(\s) > \max_{\h \in \cX \setminus A, \, \h \neq \s} \Phi(\h, A \cup \{\s\}) - H(\h).
\end{equation}
\textup{(i)} If~\eqref{eq:suffcondPE} holds for the pair $(\s,A)$, then, setting $\G:= \Phi(\s,A) - H(\s)$, we have that for any $\e>0$
\[
	\limb \pr{ e^{\b (\G-\e)} < \tha < e^{\b (\G+\e)}} =1, \qquad \text{ and } \qquad \limb \frac{1}{\b} \log \E \tha = \G.
\]
\textup{(ii)} If~\eqref{eq:suffcondAE} holds for the pair $(\s,A)$, then
\[
	\frac{\tha}{ \E \tha} \cd \rmexp(1), \quad \mathrm{ as } \, \, \binf.
\]
More precisely, there exist two functions $k_1(\b)$ and $k_2(\b)$ with $\limb k_1(\b)=0$ and $\limb k_2(\b)=0$ such that for any $s>0$
\[
	\Big | \pr{\frac{ \tha }{\E \tha} > s} - e^{-s} \Big | \leq k_1(\b) e^{-(1-k_2(\b))s}.
\]
\end{prop}
Condition~\eqref{eq:suffcondPE} says that the initial configuration $\s$ has an energy barrier separating it from the target subset $A$ that is maximum over the entire energy landscape. Informally, this means that all other ``valleys'' (or more formally \textit{cycles}, see definition in~\cite{MNOS04}) of the energy landscape are not deeper than the one where the Markov chain starts; for this reason, the authors in~\cite{NZB16} refer to~\eqref{eq:suffcondPE} as ``absence of deep cycles''. On the other hand, condition~\eqref{eq:suffcondAE} guarantees that from any configuration $\h \in \cX$ the Markov chain $\smash{\xtbb}$ reaches the set $A \cup \{\s\}$ on a time scale strictly smaller than that at which the transition from $\s$ to $A$ occurs. We remark that both these conditions are sufficient, but not necessary, see~\cite{NZB16} for further discussion. 

For the proof of Theorem~\ref{thm:mix_triangularlattice}, we will also need the following proposition, which is also a general result concerning the asymptotic behavior of mixing time and spectral gap of any Metropolis Markov chain.
\begin{prop}[Mixing time asymptotics {\cite[Proposition 3.24]{NZB16}}]\label{prop:modindep3}
For any $0 < \e < 1$
\[
	\limb \frac{1}{\b} \log t^{\mathrm{mix}}_\b(\e) = \limb -\frac{1}{\b} \log \rho_\b = \G^*,
\]
where $\G^*:= \max_{\h \in \cX, \, \h \neq \s} \Phi(\h,\s) - H(\h)$ for any stable configuration $\s \in \ss$. Furthermore, there exist two positive constants $0 < c_1 \leq c_2 < \infty$ independent of $\b$ such that for every $\b \geq 0$
\[
	c_1 e^{-\b \G^*} \leq \rho_\b \leq c_2 e^{-\b \G^*}.
\]
\end{prop}

\subsection{Proofs of Theorem \ref{thm:tabc}(i)-(iii)}
\label{newsub42}
From Theorem~\ref{thm:structuralproperties}(iii) it immediately follows that
\begin{equation}
\label{eq:tGaabbcc}
	\max_{\s \neq \aa, \bb, \cc} \Phi(\s, \{\aa,\bb,\cc\}) - H(\s)  \leq \min \{ \vdtr, 2\hdtr\}.
\end{equation}
Furthermore, we claim that the following identity holds:
\begin{equation}
\label{eq:tGbbcc}
	\max_{\s \neq \bb, \cc} \Phi(\s, \{\bb,\cc\}) - H(\s) = \min \{ \vdtr, 2\hdtr\} +1.
\end{equation}
First notice that since $\aa \in \cX \setminus \{\bb,\cc\}$, we have
\[
	\max_{\s \neq \bb, \cc} \Phi(\s, \{\bb,\cc\}) - H(\s) \geq \Phi(\aa,\{\bb,\cc\}) - H(\aa) = \min \{ \vdtr, 2\hdtr\} +1.
\]
In order to prove that identity~\eqref{eq:tGbbcc} holds, we need to show that this lower bound is sharp. In particular, we need to show that $\Phi(\s,\{\bb,\cc\}) -H(\s) \leq \min \{ \vdtr, 2\hdtr\}+1$ for every configuration $\s \neq \aa,\bb,\cc$, but we will actually prove a stronger inequality, namely
\begin{equation}
\label{eq:lowerboundtG}
	\Phi(\s,\bb) - H(\s) \leq \min \{ \vdtr, 2\hdtr\}+1, \quad \forall \, \s \in \cX \setminus \{ \aa, \bb, \cc\}.
\end{equation}
In Subsection~\ref{sub53} we introduced a iterative procedure that builds a path from any configuration $\s$ to the set of stable configuration $\ss$. More specifically, inspecting the proof of Theorem~\ref{thm:structuralproperties}(iii), we notice that every configuration $\s \neq \aa,\bb,\cc$ can be reduced either directly to $\bb$, or otherwise to $\aa$ or $\cc$, depending on its geometrical features. If $\s$ can be reduced directly to $\bb$, then we prove therein that $\Phi(\s,\bb) -H(\s) \leq \min \{ \vdtr, 2\hdtr\}$. If not, then $\s$ has to display a vertical $\aa$-- or $\cc$--bridge and $\vdtr \leq 2 \hdtr$. In the proof of Theorem~\ref{thm:structuralproperties}(iii) we construct a path $\o$ from $\s$ to $\aa$ (respectively, $\cc$) such that $\Phi_{\o} \leq H(\s) + 2$, which, concatenated with the reference path from $\aa$ to $\bb$ (exhibited in Proposition~\ref{prop:refpathaatobb}) or the analogous reference path from $\cc$ to $\bb$, shows that
$
	\Phi(\s,\bb) \leq \max\{ H(\s)+2, \Phi(\aa,\bb)\}.
$
Thus,
\begin{align*}
\Phi(\s,\bb)-H(\s)
	& \leq \max\{ 2, \Phi(\aa,\bb) - H(\s) \} \leq \max\{ 2, \Phi(\aa,\bb) - H(\aa) \} = \max\{ 2, \min\{\vdtr,2\hdtr\}+1  \} \\
	& \leq \min\{\vdtr,2\hdtr\}+1,
\end{align*}
which implies that inequality~\eqref{eq:lowerboundtG} holds. In view of~\eqref{eq:tGbbcc}, the pair $(\aa,\{\bb,\cc\})$ then satisfies condition~\eqref{eq:suffcondPE}, since
\[
	\Phi(\aa,\{\bb,\cc\}) - H(\aa) = \min \{ \vdtr, 2\hdtr\} +1 = \max_{\s \neq \bb, \cc} \Phi(\s, \{\bb,\cc\}) - H(\s),
\]
and Proposition~\ref{prop:modindep}(i) then yields statements (i) and (ii) of Theorem~\ref{thm:tabc}. Furthermore, by combining the latter identity and inequality~\eqref{eq:tGaabbcc}, we obtain
\[
	\Phi(\aa,\{\bb,\cc\}) - H(\aa) = \min \{ \vdtr, 2\hdtr\}+1 > \min \{ \vdtr, 2\hdtr\} \geq \max_{\s \neq \aa, \bb, \cc} \Phi(\s, \{\aa,\bb,\cc\}) - H(\s),
\]
and thus condition~\eqref{eq:suffcondAE} holds for the pair ($\aa$,$\{\bb,\cc\}$). Proposition~\ref{prop:modindep}(ii) then yields the asymptotic exponentiality of the rescaled tunneling time $\tabc / \E \tabc$, \ie
\begin{equation}
\label{eq:aetabc}
	\frac{\tabc}{\E \tabc} \cd \rmexp(1), \quad \text{ as } \binf,
\end{equation}
proving Theorem~\ref{thm:tabc}(iii).\\

Consider now the other tunneling time $\tab$. From inequality~\eqref{eq:lowerboundtG} it immediately follows that
\[
	\max_{\s \neq \bb} \Phi(\s, \bb) - H(\s) \leq \min \{ \vdtr, 2\hdtr\}+1,
\]
which, in view of Proposition~\ref{prop:refpathaatobb}, implies that
\[
	\Phi(\aa,\bb) -H(\aa) = \min \{ \vdtr, 2\hdtr\} +1 = \max_{\s \neq \bb} \Phi(\s, \bb) - H(\s).
\]
Hence the pair $(\aa, \{\bb\})$ satisfies condition~\eqref{eq:suffcondPE} and statements (i) and (ii) of Theorem~\ref{thm:tabc} for the tunneling time $\tab$ immediately follow from Proposition~\ref{prop:modindep}(i). \qed

\subsection{Proof of Theorem \ref{thm:mix_triangularlattice}}
\label{newsub43}
The proof readily follows from Proposition~\ref{prop:modindep3}, since by combining inequality~\eqref{eq:lowerboundtG} and Theorem~\ref{thm:structuralproperties}(ii) we get
\[
	\G^* = \max_{\s \neq \bb} \Phi(\s,\bb)- H(\s) = \min \{ K, 2L \} +1. \hfill \qed
\]

\subsection{Asymptotic exponentiality of the tunneling time \texorpdfstring{$\tab$}{}} 
\label{newsub44}

The pair $(\aa,\{\bb\})$ does not satisfy condition~\eqref{eq:suffcondAE}, due to the presence of a deep cycle (the one where configuration $\cc$ lies) different from the initial configuration $\aa$ lies. Indeed, 
$
	\Phi(\aa,\bb) - H(\aa) \not < \Phi(\cc,\bb) - H(\cc),
$
as shown in Theorem~\ref{thm:structuralproperties}(ii). Hence, the proof of Theorem~\ref{thm:tabc}(iv) does not follow from the general results outlined in Proposition~\ref{prop:modindep}, as in the case of statement (iii), but leverages in a crucial way the structure of the state space $\cX$.

In view of the intrinsic symmetry of a triangular grid $\L$, it is intuitive that the energy landscape $\cX$ on which the Markov chain $\smash{\xtbb}$ evolves is highly symmetric, as witnessed by Figure~\ref{fig:ss4x6}. In this subsection, we show that the $2K\times 3L$ triangular grid has nontrivial automorphisms and discuss the consequences of this fact for the state space $\cX$.
We then leverage these symmetries to derive properties for the tunneling time $\tab$ (Proposition~\ref{prop:transitiontimeproperties}) and a stochastic representation for this latter (Corollary~\ref{cor:stochasticrepresentationTxy}), and, ultimately, to prove the asymptotic exponentiality of the scaled hitting time $\t^{\aa}_{\bb} / \E \t^{\aa}_{\bb}$ in the limit $\binf$, \ie Theorem~\ref{thm:tabc}(iv).

For every $k=0,\dots,6L-1$, the \textit{axial symmetry with respect to column} $c_k$ is the permutation $\xi_k: V \to V$ that maps site $(i,j)$ into site $(i, 2k-j)$, see Figure~\ref{fig:automorphism_triangular_grid} for an example. 
In any such axial symmetry neighboring sites are mapped into neighboring sites, namely any pair of sites $u,v$ form an edge if and only if the sites $\xi_k(u),\xi_k(v)$ form an edge. Hence, $\xi_k$ is an automorphism of the graph $\L$ for every $k=0,\dots,6L-1$. Each of these axial symmetries swaps two of the three components while mapping the third one to itself. Specifically,
\begin{equation}
\label{eq:xik}
	\begin{cases}
	\xi_k(\AA)=\AA, \, \xi_k(\BB)=\CC, \, \xi_k(\CC)=\BB, & \text{ if } k \equiv 0 \pmod 3,\\
	\xi_k(\AA)=\CC, \, \xi_k(\BB)=\BB, \, \xi_k(\CC)=\AA, & \text{ if } k \equiv 1 \pmod 3,\\
	\xi_k(\AA)=\BB, \, \xi_k(\BB)=\AA, \, \xi_k(\CC)=\CC, & \text{ if } k \equiv 2 \pmod 3.
	\end{cases}
\end{equation}
\vspace{-0.65cm} 
\begin{figure}[!ht]
	\centering
	\subfloat{\includegraphics{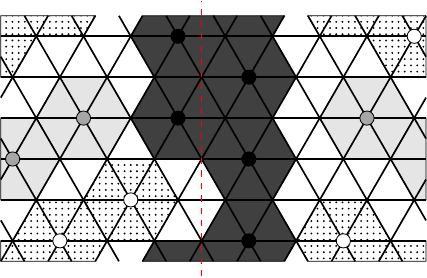}}
	\hspace{1cm}
	\subfloat{\includegraphics{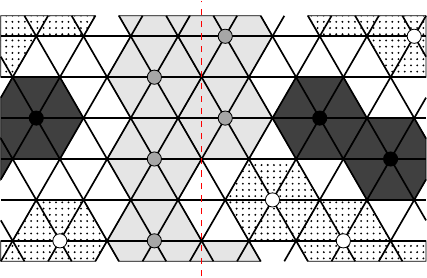}}
	\caption{A hard-core configuration on the $6 \times 9$ triangular grid $\L$ before (left) and after (right) the axial symmetry with respect to column $c_8$, highlighted as dashed red vertical line, which maps column $c_j$ into $c_{16 - j}$ for every $j=0,\dots,16$ and column $c_{17}$ into itself. This axial symmetry induces an automorphism $\xi_{\aa,\bb}$ of the triangular grid $\L$ and, in particular, maps gray sites into black sites (and vice-versa) while white sites are only permuted.}
	\label{fig:automorphism_triangular_grid}
\end{figure}
\FloatBarrier


As illustrated by the next proposition, these axial symmetries of $\L$ induce automorphisms of the state space diagram $\cX$ corresponding to the hard-core dynamics on $\L$. Hence, the state space $\cX$ is highly symmetric, as clearly visible in Figure~\ref{fig:ss4x6}, which shows the state space diagram of the hard-core model on the $4 \times 6$ triangular grid $\L$. Leveraging the symmetry of $\cX$, we construct a coupling between different copies of the Markov chain $\smash{\xtbb}$ and prove in this way properties of the first hitting time $\t^\aa_{\{\bb,\cc\}}$.

\begin{prop}[Tunneling time properties] \label{prop:transitiontimeproperties}
Let $\smash{\xtbb}$ be the Metropolis Markov chain corresponding to the hard-core dynamics on the $2K \times 3L$ triangular grid. Then, for every $\b > 0$,
\begin{itemize}
	\item[\textup{(i)}] The random variable $\smash{X_{\t^\aa_{\{\bb,\cc\}}}}$ has a uniform distribution over $\{\bb,\cc\}$;
	\item[\textup{(ii)}] $\smash{\t^\aa_{\{\bb,\cc\}} \ed \t^\bb_{\{\aa,\cc\}} \ed \t^\cc_{\{\aa,\bb\}}}$;
	\item[\textup{(iii)}] The random variables $\t^\aa_{\{\bb,\cc\}}$ and $X_{\t^\aa_{\{\bb,\cc\}}}$ are independent.
\end{itemize}
\end{prop}

\begin{proof}
For the purpose of this proof it is enough to consider three axial symmetries that cover the three cases in~\eqref{eq:xik} and thus we denote $\xi_{\bb,\cc}:=\xi_0$, $\xi_{\aa,\cc}:=\xi_1$, and $\xi_{\aa,\bb}:=\xi_2$.
The automorphism $\xi_{\bb,\cc}: V \to V$ induces a permutation $\overline{\xi}$ of the collection $\cX$ of hard-core configurations on $\L$. More precisely, $\overline{\xi}$ maps the hard-core configuration $\s \in \cX$ into a new configuration $\overline{\xi}(\s)$ defined as $(\overline{\xi}(\s))(v) = \s(\xi_{\bb,\cc}(v))$ for every $v \in V$.
In fact, $\overline{\xi}$ is an automorphism of the state space diagram, seen as a graph with vertex set $\cX$ and such that any pair of hard-core configurations $\s,\s' \in \cX$ is connected by an edge if and only if $\s$ and $\s'$ differ in no more than one site, \ie $\|\s-\s'\|\leq 1$. By construction,
\begin{equation}
\label{eq:automorphism_abc}
	\overline{\xi}(\bb)=\cc, \quad \overline{\xi}(\cc)=\bb, \quad \text{ and } \quad \overline{\xi}(\aa)= \aa.
\end{equation}
Assume the Metropolis Markov chain $\smash{\xtbb}$ on $\L$ starts in configuration $\aa$ at time $0$. Let $\smash{\{Y^\b_t\}_{t\in \N}}$ be the Markov chain that mimics the moves of the Markov chain $\smash{\xtbb}$ via the automorphism $\overline{\xi}$, \ie set $\smash{Y^\b_t:=\overline{\xi} \big (X^\b_t\big )}$ for every $t \in \N$. 
For notational compactness, we suppress in this proof the dependence on $\b$ of these two Markov chains. For any pair of hard-core configurations $\s,\s' \in \cX$, any transition of the chain $Y^\b_t$ from $\h=\overline{\xi}(\s)$ to $\h'=\overline{\xi}(\s')$ is feasible and occurs with the same probability as the transition from $\s$ to $\s'$, since $\overline{\xi}$ is an automorphism. Therefore, the Markov chains $\{X_t\}_{t\in \N}$ and $\{Y_t\}_{t\in \N}$ are two copies of the hard-core dynamics on $\L$ living in the same probability space, and we have then defined in this way a coupling between them. In view of~\eqref{eq:automorphism_abc}, this coupling immediately implies that the Markov chain $\{X_t\}_{t\in \N}$ started at $\aa$ hits configuration $\bb$ precisely when the chain $\{Y_t\}_{t\in \N}$ hits $\cc$. 
Hence,
\begin{equation}
\label{eq:jointdistribution}
	\pr{X_{\tau^\aa_{\{\bb,\cc\}}}=\bb , \, \tau^\aa_{\{\bb,\cc\}} \leq t}  = \pr{\overline{\xi}(X_{\tau^\aa_{\{\bb,\cc\}}})=\overline{\xi}(\bb) , \tau^{\overline{\xi}(\aa)}_{\overline{\xi}(\{\bb,\cc\})} \leq t} = \pr{Y_{\tau^\aa_{\{\bb,\cc\}}}=\cc , \, \tau^\aa_{\{\bb,\cc\}} \leq t}.
\end{equation}
Taking the limit $t \to \infty$ in~\eqref{eq:jointdistribution}, we obtain
\[
	\pr{X_{\tau^\aa_{\{\bb,\cc\}}}=\bb}  = \pr{Y_{\tau^\aa_{\{\bb,\cc\}}}=\cc}.
\]
Using the fact that $\{X_t\}_{t\in \N}$ and $\{Y_t\}_{t\in \N}$ have the same statistical law, being two copies of the same Markov chain, it then follows that the random variable $\smash{X_{\tau^\aa_{\{\bb,\cc\}}}}$ has a uniform distribution over $\{\bb,\cc\}$, that is property (i). In particular,
\begin{equation}\label{eq:uniformdistribution}
	\pr{X_{\tau^\aa_{\{\bb,\cc\}}}=\bb}=\frac{1}{2}.
\end{equation}
Let $\hat{\xi}$ be the permutation of $\cX$ induced by the automorphism $\xi_{\{\aa,\cc\}} \circ \xi_{\{\aa,\bb\}}$. Constructing the coupling using $\hat{\xi}$ and arguing as above, we can deduce that 
\[
\pr{X_{\tau^\aa_{\{\bb,\cc\}}}=\bb , \, \tau^\aa_{\{\bb,\cc\}} \leq t}  = \pr{\hat{\xi}(X_{\tau^\aa_{\{\bb,\cc\}}})=\hat{\xi}(\bb) , \tau^{\hat{\xi}(\aa)}_{\hat{\xi}(\{\bb,\cc\})} \leq t} = \pr{Y_{\tau^\bb_{\{\cc,\aa\}}}=\cc , \, \tau^\bb_{\{\cc,\aa\}} \leq t},
\]
and
\[
	\pr{X_{\tau^\aa_{\{\bb,\cc\}}}=\cc , \, \tau^\aa_{\{\bb,\cc\}} \leq t}  = \pr{\hat{\xi}(X_{\tau^\aa_{\{\bb,\cc\}}})=\hat{\xi}(\cc) , \tau^{\hat{\xi}(\aa)}_{\hat{\xi}(\{\bb,\cc\})} \leq t} = \pr{Y_{\tau^\bb_{\{\cc,\aa\}}}=\aa , \, \tau^\bb_{\{\cc,\aa\}} \leq t}.
\]
Summing side by side these latter two identities yields that for every $t \geq 0$
\[
	\pr{\tau^\aa_{\{\bb,\cc\}} \leq t}  = \pr{\tau^\bb_{\{\cc,\aa\}} \leq t},
\]
proving property (ii). Note that
\begin{equation}
\label{eq:partialidentity}
	\pr{\tau^\aa_{\{\bb,\cc\}} \leq t}  =  \pr{X_{\tau^\aa_{\{\bb,\cc\}}}=\bb , \, \tau^\aa_{\{\bb,\cc\}} \leq t} +  \pr{X_{\tau^\aa_{\{\bb,\cc\}}}=\cc , \, \tau^\aa_{\{\bb,\cc\}} \leq t} = 2 \cdot \pr{X_{\tau^\aa_{\{\bb,\cc\}}}=\bb , \, \tau^\aa_{\{\bb,\cc\}} \leq t},
\end{equation}
where the last passage follows from~\eqref{eq:jointdistribution} using again the fact that $\{X_t\}_{t\in \N}$ and $\{Y_t\}_{t\in \N}$ have the same statistical law. Combining identities~\eqref{eq:uniformdistribution} and~\eqref{eq:partialidentity}, we obtain that for every $t\geq 0$,
\[
	\pr{X_{\tau^\aa_{\{\bb,\cc\}}}=\bb , \, \tau^\aa_{\{\bb,\cc\}} \leq t} = \pr{X_{\tau^\aa_{\{\bb,\cc\}}}=\bb} \cdot \pr{\tau^\aa_{\{\bb,\cc\}} \leq t},
\]
that is property (iii).
\end{proof}

The next corollary shows how the symmetries of the hard-core dynamics on a triangular grid $\L$ derived in Proposition~\ref{prop:transitiontimeproperties} can be used to obtain a stochastic representation for the tunneling time $\tab$, that will be crucial to prove Theorem~\ref{thm:tabc}(iv). The underlying idea is that, on the time-scale at which the transition from $\aa$ to $\bb$ occurs, the evolution of $\smash{\xtbb}$ can be represented by a $3$-state Markov chain with a complete graph as state space diagram whose states correspond to the three valleys/cycles around the stable configurations $\aa$, $\bb$, and $\cc$. Similar ideas have been successfully used to describe metastability and tunneling phenomena in~\cite{BeltranLandim15,Landim2016a,Landim2016,Zocca2018}.

\begin{cor}[Stochastic representation of the tunneling time $\tab$] \label{cor:stochasticrepresentationTxy}
Let $\{\t^{(i)}\}_{i \in \N}$ be a sequence of i.i.d.~random variables with common distribution $\smash{\tau \ed \tabc}$ and $\mathcal{G}$ an independent geometric random variable with success probability $1/2$, namely $\prin{\mathcal{G} = m} =2^{-m}$, for $m \geq 1$. Then,
\begin{equation}
\label{eq:geometricrepresentationTxy}
	\tab \, \ed \sum_{i=1}^{\mathcal{G}} \t^{(i)},
\end{equation}
and, in particular, $\E \tab = 2 \cdot \E \tabc$. Furthermore, if additionally there exists a non-negative random variable $Y$ such that $\smash{ \t /\E \t  \cd Y}$ as $\binf$, then
\begin{equation}
\label{eq:aetxy}
	\frac{\tab}{\E \tab} \cd \frac{1}{\E \mathcal{G}} \sum_{i=1}^{\mathcal{G}} Y^{(i)}, \quad \mathrm{ as } \, \, \binf,
\end{equation}
where $\{Y^{(i)}\}_{i \in \N}$ is a sequence of i.i.d.~random variables distributed as $Y$.
\end{cor}
\begin{proof}
Let $\mathcal{G}$ be the random variable counting the number of non-consecutive visits of the Markov chain to $\{\aa,\cc\}$ until $\bb$ is hit for the first time (counting the initial configuration $\aa$ as first visit). In view of Proposition~\ref{prop:transitiontimeproperties}(i), the random variable $\mathcal{G}$ is geometrically distributed with success probability $\frac{1}{2}$, with distribution $\prin{\mathcal{G} = m} =2^{-m}$, for $m \geq 1$. In particular, $\mathcal{G}$ it does not depend on the inverse temperature $\b$. The amount of time it takes for the Markov chain started in a stable configuration to hit any of the other two stable configurations does not depend on the initial stable configuration, by virtue of Proposition~\ref{prop:transitiontimeproperties}(ii). In view of these considerations and using the independence property in Proposition~\ref{prop:transitiontimeproperties}(iii), we deduce the stochastic representation~\eqref{eq:geometricrepresentationTxy} for the tunneling time $\tab$. The identity $\E \tab = 2 \cdot \E \tabc$ then immediately follows from Wald's identity, since both $\mathcal{G}$ and $\tabc$ have finite expectation and $\E \mathcal{G} = 2$.

Lastly, we turn to the proof of the limit in distribution~\eqref{eq:aetxy}. Denoting by $\LT_{A}(s)=\E (e^{-s A})$, with $s \geq 0$, the Laplace transform of a random variable $A$, the stochastic representation~\eqref{eq:geometricrepresentationTxy} yields
$
	\LT_{\tab} = G_{\mathcal{G}}\left(\LT_{\t}(s)\right),
$
where $G_{\mathcal{G}}(\cdot)$ is the probability generating function of the random variable $\mathcal{G}$, \ie $G_{\mathcal{G}}(z)=\E (z^{\mathcal{G}})$ for every $z \in [0,1]$. By assumption $\LT_{\t /\E \t } (s) \to \LT_Y(s)$ as $\binf$. Using the fact that $\E \tab  = \E \t \cdot \E \mathcal{G}$ we obtain
\[
	\LT_{\tab /\E\tab}  = G_{\mathcal{G}} \left(\LT_{\t/ \E \t}(s / \E \mathcal{G})\right) \stackrel{\binf}{\longrightarrow} G_{\mathcal{G}} \left(\LT_{Y}(s / \E \mathcal{G})\right),
\]
and the continuity theorem for Laplace transforms yields the conclusion.
\end{proof}

\begin{proof}[Proof of Theorem~\ref{thm:tabc}\textup{(iv)}] 

Corollary~\ref{cor:stochasticrepresentationTxy} yields
\[
	\frac{\tab}{\E \tab} \cd \frac{1}{2} \sum_{i=1}^{\rmgeo(1/2)} Y^{(i)},  \quad \text{ as } \binf,
\]
where $\{Y^{(i)}\}_{i \in \N}$ are i.i.d.~exponential random variables, in view of~\eqref{eq:aetabc}. The statement in Theorem~\ref{thm:tabc}(iv) then follows by noticing that a geometric sum of i.i.d.~exponential random variables scaled by its mean is also exponentially distributed with unit mean. 
\end{proof}

\textbf{Acknowledgments} The author has been supported by NWO grants 639.033.413 and 680.50.1529 and is grateful to F.R.~Nardi, S.C.~Borst, and  J.S.H.~van Leeuwaarden for the precious feedback on this work.

\end{document}